\documentclass[11pt,reqno,oneside]{amsart}

\usepackage{graphicx}
\usepackage{amssymb}
\usepackage{epstopdf}

\usepackage[a4paper, total={6in, 9in}]{geometry}

\usepackage{amsmath,amsfonts,amsthm,mathrsfs,amssymb,cite}
\usepackage[usenames]{color}

\newtheorem{thm}{Theorem}[section]

\newtheorem{lem}{Lemma}[section]

\theoremstyle{definition}

\theoremstyle{remark}

\numberwithin{equation}{section}

\numberwithin{equation}{section}

\newcounter{saveeqn}
\newcommand{\eqnref}[1]{(\ref {#1})}

\newcommand{\Bf}{\mathbf{f}}
\newcommand{\Bh}{\mathbf{h}}

\newcommand{\Bz}{\mathbf{z}}

\newcommand{\bH}{\mathbf{H}}

\newcommand{\Bx}{\mathbf{x}}
\newcommand{\By}{\mathbf{y}}

\newcommand{\bZ}{\mathbf{z}}

\newcommand{\BB}{\mathbf{B}}

\newcommand{\BN}{\mathbf{N}}

\newcommand{\Gs}{\sigma}

\newcommand{\tdx}{\tilde{\Bx}}
\newcommand{\tdy}{\tilde{\By}}

\newcommand{\Acal}{\mathcal{A}}
\newcommand{\Kcal}{\mathcal{K}}
\newcommand{\Lcal}{\mathcal{L}}
\newcommand{\Scal}{\mathcal{S}}

\newcommand{\Mcal}{\mathcal{M}}
\newcommand{\Ncal}{\mathcal{N}}

\newcommand{\Ocal}{\mathcal{O}}

\newcommand{\Pcal}{\mathcal{P}}
\newcommand{\Qcal}{\mathcal{Q}}

\newcommand{\ds}{\displaystyle}
\newcommand{\la}{\langle}
\newcommand{\ra}{\rangle}

\newcommand{\RR}{\mathbb{R}}

\newcommand{\p}{\partial}

\newcommand{\beq}{\begin{equation}}
\newcommand{\eeq}{\end{equation}}

\DeclareMathAlphabet{\itbf}{OML}{cmm}{b}{it}

\def\bu{{{\mathbf{u}}}}

\title[Identifying magnetized anomalies in MHD model]{On identifying magnetized anomalies using geomagnetic monitoring II. A Magnetohydrodynamic Model}

\author{Youjun Deng}
\address{School of Mathematics and Statistics, Central South University, Changsha, Hunan, China.}
\email{youjundeng@csu.edu.cn, dengyijun\_001@163.com}

\author{Jinhong Li}
\address{School of Science, Qilu University of Technology (Shandong Academy of Sciences), Jinan, Shandong, China}
\email{lijinhong@qlu.edu.cn}

\author{Hongyu Liu}
\address{Department of Mathematics, Hong Kong Baptist University, Kowloon, Hong Kong SAR, China}
\email{hongyu.liuip@gmail.com, hongyuliu@hkbu.edu.hk}

\date{} % Activate to display a given date or no date (if empty),
\begin{document}
\maketitle

\begin{abstract}

This paper is a continuation and an extension of our recent work \cite{DLL:18} on the identification of magnetized anomalies using geomagnetic monitoring, which aims to establish a rigorous mathematical theory for the geomagnetic detection technology. Suppose a collection of magnetized anomalies is presented in the shell of the Earth. By monitoring the variation of the magnetic field of the Earth due to the presence of the anomalies, we establish sufficient conditions for the unique recovery of those unknown anomalies. In \cite{DLL:18}, the geomagnetic model was described by a linear Maxwell system. In this paper, we consider a much more sophisticated and complicated magnetohydrodynamic model, which stems from the widely accepted dynamo theory of geomagnetics.

%This paper concerns with identifying the magnetized anomalies under Magnetohydrodynamics model. It is an extension of the work \cite{DLL:18}, where a linear Maxwell model is considered. As the Magnetohydrodynamics model is a coupled nonlinear system, which is quite different from that in \cite{DLL:18}, new technical strategy is used for the unique recovery. By using the difference of the magnetic fields before and after the presence of the magnetized anomalies, we show that one can uniquely recover the locations as well as their material parameters, magnetic permeability and electrical conductivity, of the anomalies.

\medskip



\end{abstract}

\section{Introduction}

\subsection{Background on the dynamo theory of geomagnetics}

The dynamo theory proposes a mechanism such that a rotating, convecting and electrically conducting fluid can maintain the magnetic field of a celestial body. Earth's magnetic field has been interpreted by the dynamo theory through a magnetohydrodynamic (MHD) model \cite{ZZZ11,BPCo96,CZLSZ03,Hol,Lab,Rob,Weiss,Sak,ZSc00,Wik2}. Following the discussion in \cite{CZZ06,CZLSZ03,HeZo18,ZZZ11,ZSc00}, we briefly introduce the MHD system for our study. It is widely accepted that the Earth is of a core-shell structure. The dynamo mechanism states that the convecting currents of fluid metal in the Earth's outer core, driven by heat flow from the outer core, organized into rolls by the Coriolis force, create circulating electric currents, which generate the magnetic field. To describe this physical process, there are three basic partial differential equations including the Navier-Stokes equations, the heat equation and the Maxwell equations, and they are coupled together to form the MHD system.

Let $\Sigma_i, \Sigma_o$ and $\Sigma_s$ be bounded simply-connected $C^2$ domains in $\mathbb{R}^3$, respectively, signifying the inner core, outer core and shell of the Earth. $\Sigma_c:=\Sigma_o\cup\overline{\Sigma_i}$ and $\Sigma:=\Sigma_s\cup\overline{\Sigma_c}$, are also simply-connected $C^2$ domains, which denote, respectively, the Earth's core and the Earth. $\mathbb{R}^3\backslash\overline{\Sigma}$ is the outer space of the Earth. Let $\mathbf{u}(\Bx, t)$ and $p(\Bx, t)$, $(\Bx, t)\in\Sigma_o\times\mathbb{R}_+$, respectively denote the velocity field and pressure of the fluid metal in the outer core. Throughout, $\Bx\in\mathbb{R}^3$ represents the position vector and $t$ represents the time variable. Let $\Theta(\Bx, t)$, $(\Bx, t)\in\Sigma_o\times\mathbb{R}_+$, signify the deviation of the temperature from its static distribution of the Earth's outer core. Finally, we let $\BB(\Bx, t)$ and $\mathcal{H}(\Bx, t)$, $(\Bx, t)\in\mathbb{R}^3\backslash\overline{\Sigma_i}\times\mathbb{R}_+$ respectively, be the induced magnetic field and the intensity of the magnetic field. There holds $\BB(\Bx, t)=\mu(\Bx)\mathcal{H}(\Bx, t)$, $\Bx\in\mathbb{R}^3\backslash\overline{\Sigma_i}$, where $\mu(\Bx)$ is the magnetic permeability of the medium.

The motion of the fluid metal in the outer core is described by the following Navier-Stokes equations
\begin{equation}\label{eq:geom1}
\dot{\bu}+\bu\cdot\nabla \bu-\eta\Delta \bu+2\widetilde{\omega}\times \bu+\frac{1}{\varrho}\nabla p=\gamma \alpha \Theta \Bx+\frac{1}{\varrho\mu}(\nabla\times \BB)\times \BB,\ \ (\Bx, t)\in\Sigma_o\times\mathbb{R}_+,
\end{equation}
where and also in what follows, the overdot signifies the partial derivative with respect to the time variable $t$. Furthermore, the incompressibility of the fluid yields that
\begin{equation}\label{eq:geom4}
\nabla\cdot \bu(\Bx, t)=0,\quad (\Bx, t)\in\Sigma_o\times\mathbb{R}_+.
\end{equation}
The temperature distribution $\Theta(\Bx, t)$ satisfies the following heat equation
\begin{equation}\label{eq:geom3}
\dot{\Theta}+\bu\cdot\nabla \Theta=\kappa\Delta\Theta+\beta \bu\cdot \Bx,\quad (\Bx, t)\in\Sigma_o\times\mathbb{R}_+.
\end{equation}
The physical meaning of the coefficient parameters in \eqref{eq:geom1} and \eqref{eq:geom3} are given as follows. $\varrho$ is the density, $\kappa$ is the thermal diffusivity, $\alpha$ is the thermal expansion coefficient, $\eta$ is the kinematic viscosity and $\lambda$ is the magnetic diffusivity. There holds,
\begin{equation}\label{eq:ddd1}
\lambda=\frac{1}{\mu \sigma},
\end{equation}
where $\mu$ and $\Gs$ are respectively the magnetic permeability and electric conductivity. $\widetilde{\omega}$ signifies the uniform angular velocity of the rotation of the Earth's core. $\gamma$ is a positive constant which fulfils $\mathbf{g}=-\gamma \Bx$, where $\mathbf{g}$ is the gravitational field of the Earth. The parameter $\beta$ is a positive constant representing the uniformly distributed heat source. To complete the description, one needs to impose suitable boundary and initial conditions for the two partial differential equations \eqref{eq:geom1} and \eqref{eq:geom3}. There are two types of widely used boundary conditions. In this paper, we assume that the inner- and outer-bounding surfaces are nonslip and impenetrable, namely,
\beq\label{eq:geobndcon01}
\bu(\Bx,t)=0, \quad (\Bx, t)\in \, \p \Sigma_o\times \RR_+.\medskip
\eeq
One may also assume that the velocity on the boundary, $\p \Sigma_o$, is stress-free and impenetrable. The temperature boundary condition usually does not play a major role, and one may assume that the heat does not convect in the region outside the outer core, i.e.,
\beq\label{eq:geobndcon02}
\frac{\p \Theta(\Bx, t)}{\p \nu}=0. \quad (\Bx, t)\in \, \p \Sigma_o\times \RR_+,
\eeq
where $\nu$ denotes the exterior unit normal vector to $\partial\Sigma_o$. The initial conditions are prescribed as follows,
 \beq\label{eq:pssapp02}
\bu(\Bx, 0)= \Bf(\Bx), \quad \Theta(\Bx, 0)= T(\Bx), \quad \Bx\in\Sigma_o.
\eeq
Here, we would like to emphasize that the boundary and initial conditions shall not play an essential role in our study of the inverse problem on the geomagnetic detection as long as the corresponding forward problem is well-posed. We shall further remark this point in our subsequent discussion.

The induced magnetic field is governed by the following exterior problem associated with a Maxwell system,
\begin{equation}\label{eq:pssapp}
\begin{cases}
\dot{ \mathcal{H}}(\Bx,t)=\nabla\times(\bu(\Bx,t)\times \mathcal{H}(\Bx,t))-\lambda_o(\Bx)\nabla\times\nabla\times \mathcal{H}(\Bx,t),&\hspace*{-.5cm}(\Bx,t)\in\, \Sigma_o\times \RR_+, \medskip \\
\dot{\mathcal{H}}(\Bx,t)=-\lambda_s(\Bx)\nabla\times\nabla\times \mathcal{H}(\Bx,t), &\hspace*{-.5cm}(\Bx,t)\in\, \Sigma_s\times \RR_+, \medskip \\
\nabla\times\nabla\times \mathcal{H}(\Bx,t)=0,   &\hspace*{-.5cm}(\Bx,t)\in\, (\RR^3\setminus{\overline{\Sigma}})\times \RR_+,\medskip \\
\nabla\cdot (\mu(\Bx)\mathcal{H}(\Bx,t))=0, \ (\Bx,t)\in\, (\RR^3\backslash\overline{\Sigma_i})\times \RR_+; \ \mathcal{H}(\Bx, 0)= \Bh(\Bx)\chi(\Sigma_o), & \hspace*{-.5cm}\Bx\in \RR^3\setminus{\overline{\Sigma_i}},\medskip\\
\nu\times\mathcal{H}(\Bx, t)=0,\ (\Bx, t)\in \partial\Sigma_i\times\mathbb{R}_+; \ \ \mathcal{H}(\Bx,t)=\Ocal(\|\Bx\|^{-2})\ \mbox{as}\ \ \|\Bx\|\rightarrow \infty,
\end{cases}
\end{equation}
where $\lambda_o$ and $\lambda_s$ are respectively the magnetic diffusivities in the outer core and shell of the Earth. In \eqref{eq:pssapp}, we impose a perfect magnetic conducting (PMC) condition on the inner boundary $\Sigma_i$. As emphasized earlier, this boundary condition shall not play an essential role in our subsequent inverse problem study as long as the MHD system \eqref{eq:geom1}--\eqref{eq:pssapp} is well-posed in the sense that shall be prescribed later. It can be replaced by any other suitable boundary condition depending on the availability of the physical property of the inner core. In what follows, for expositional convenience, we refer to $\mathcal{H}$ as the magnetic field.

\subsection{Mathematical formulation of the geomagnetic detection}

We present the mathematical formulation of the geomagnetic detection.  Let $\mu$ and $\sigma$ be real-valued $L^\infty$ functions in $\mathbb{R}^3\backslash\overline{\Sigma_i}$, such that $\mu$ is positive and $\sigma$ is nonnegative.
Let $\mu_0$ denote the permeability of the uniformly homogeneous outer space $\mathbb{R}^3\backslash\overline{\Sigma}$. The material distribution in $\mathbb{R}^3\backslash\overline{\Sigma_i}$ is described by
\begin{equation}\label{eq:gome5}
\Gs(\Bx)=\Gs_c(\Bx)\chi(\Sigma_o)+\Gs_s(\Bx)\chi(\Sigma_s),\ \ \quad\mu(\Bx)=(\mu_c(\Bx)-\mu_0)\chi(\Sigma_o)+\mu_0, \ \ \Bx\in \mathbb{R}^3\backslash\overline{\Sigma_i},
\end{equation}
where and also in what follows, $\chi$ denotes the characteristic function. By \eqref{eq:gome5}, we know that the mediums in the outer core and the shell of the Earth are respectively characterized by $(\Sigma_o; \mu_c,\sigma_c)$ and $(\Sigma_s; \mu_0,\sigma_s)$.
Define by $\lambda$ the magnetic diffusivity in the Earth's interior, then it satisfies
\beq\label{eq:gome0101}
\lambda(\Bx)=\lambda_c(\Bx)\chi(\Sigma_o)+\lambda_s(\Bx)\chi(\Sigma_s),
\eeq
where $\lambda_c(\Bx)=\frac{1}{\mu_c\sigma_c}$ and $\lambda_s(\Bx)=\frac{1}{\mu_0\sigma_s}$. Henceforth, we let $\mathbf{u}_0$, $p_0$, $\Theta_0$, and $\mathcal{H}_0$, respectively, denote the velocity field, pressure function, temperature distribution and magnetic field of the MHD system \eqref{eq:geom1}--\eqref{eq:pssapp} associated with the medium configuration in \eqref{eq:gome5} and \eqref{eq:gome0101}.

Next we suppose that a collection of magnetized anomalies presented in the shell of the Earth. Let $D_l$, $l=1,2,\ldots, l_0$, denote the magnetized anomalies, where $D_l$, $1\leq l\leq l_0$ are simply-connected $C^2$ domains such that
%\begin{figure}%[h]
%\begin{center}
 % \includegraphics[width=2.3in,height=1.8in]{Earth.jpg}
  %\end{center}
 %\caption{Schematic illustration of identifying magnetized anomalies using the geomagnetic monitoring. \label{fig1}}
%\end{figure}
the corresponding material parameters are given by $\mu_l$ and $\sigma_l$. It is assumed that $\mu_l$ and $\sigma_l$ are all positive constants with $\mu_l\neq \mu_0$, $1\leq l\leq l_0$. With the presence of the magnetized anomalies $(D_l; \mu_l, \sigma_l)$, $l=1,2,\ldots,l_0$, in the shell of the Earth, the medium configuration in $\mathbb{R}^3\backslash\overline{\Sigma_i}$ is then described by
\beq\label{eq:paradefsigma}
\begin{split}
\tilde\Gs(\Bx)=& \Gs_c(\Bx)\chi(\Sigma_o)+\Gs_s(\Bx)\chi(\Sigma_s)+\sum_{l=1}^{l_0} \sigma_l\chi(D_l),\\
\quad\tilde\mu(\Bx)=& (\mu_c(\Bx)-\mu_0)\chi(\Sigma_o)+\sum_{l=1}^{l_0} (\mu_l-\mu_0)\chi(D_l)+\mu_0,\\
\end{split}
\eeq
Accordingly, the magnetic diffusivity $\tilde\lambda$ is given by
\begin{equation}\label{eq:mmm1}
\tilde\lambda(\Bx)=\lambda_c(\Bx)\chi(\Sigma_o)+\lambda_s(\Bx)\chi(\Sigma_s\setminus{\overline{\bigcup_{l=1}^{l_0}D_l}})+\sum_{l=1}^{l_0}\lambda_l(x)\chi(D_l),
\end{equation}
where $\lambda_l(\Bx)=\frac{1}{\mu_l\sigma_l}$, $l=1, 2, \ldots, l_0.$ Let $\widetilde{\mathbf{u}}$, $\widetilde{p}$, $\widetilde{\Theta}$, and $\widetilde{\mathcal{H}}$, respectively, denote the velocity field, pressure function, temperature distribution and magnetic field of the MHD system \eqref{eq:geom1}--\eqref{eq:pssapp} associated with the medium configuration in \eqref{eq:paradefsigma} and \eqref{eq:mmm1}.

Let $\widetilde{\Gamma}$ be a simply-connected and analytic surface which encloses $\Sigma$, and let $\Gamma$ be an open patch of $\widetilde{\Gamma}$.
In the current article, we are mainly concerned with the following inverse problem,
\begin{equation}\label{eq:geom6}
\left(\widetilde{\mathcal{H}}(\Bx, t)-\mathcal{H}_0(\Bx, t) \right)\bigg|_{(\Bx,t)\in\Gamma\times\mathbb{R}_+}\longrightarrow \bigcup_{l=1}^{l_0} (D_l;\mu_l,\sigma_l).
\end{equation}
That is, we make use the variation of the geomagnetic field due to the presence of the magnetized anomalies in the shell of the Earth to recover the locations as well as the material parameters of the anomalies. For simplicity, one may take $\Gamma$ to be an open patch of $\partial B_R$ with $B_R$ a sufficiently large central ball containing $\Sigma$. Several remarks are in order regarding the magnetic anomaly detection problem described above.

First, from a practical point of view, the geomagnetic configuration in the core of the Earth, namely the fluid metal, should not be assumed to be a-priori known. In fact, the movement of the fluid metal in the outer core is to generate the source input for the Maxwell system \eqref{eq:pssapp}. On the other hand, the measurement data $\mathcal{H}_0$ clearly encode the information of this generated source. This observation is critical for our subsequent unique recovery study of the geomagnetic detection. Second, the time interval for the measurement in \eqref{eq:geom6} is $\mathbb{R}_+$. However, it can actually be replaced by a finite time-interval $[0, T_0]$ such that after $T_0$, the interrupted geomagnetic field due to the presence of the anomalies leaves away from the Earth. This is clearly a physically reasonable assumption. Third, the measurement surface $\Gamma$ is in a scale much smaller than the Earth, and we are mainly concerned with the region under $\Gamma$ such that the geomagnetic effect can reach $\Gamma$. Hence, it is unobjectionable to assume that the medium configuration in the rest part of the Earth is the same as that in the region under $\Gamma$. One may consider the submarine detection using the geomagnetic monitoring and it can be assumed that the magnetic diffusivity $\lambda_s$ is the one for the sea. Throughout the rest of the paper, we assume that $\lambda_s$ is a constant and it is known a-priori.
Finally, it is noted that the anomalies are also in a size much smaller than that of the Earth.

According to our discussion above, it is sufficient for us to consider the inverse problem \eqref{eq:geom6} associated with the magnetohydrodynamic Maxwell system \eqref{eq:pssapp} in $\mathbb{R}^3\backslash\overline{\Sigma_c}$, with the implicit use that the magnetic fields are generated by the outer core of the Earth through the MHD system \eqref{eq:geom1}--\eqref{eq:pssapp}. The well-posedness of the complicated nonlinear MHD system \eqref{eq:geom1}--\eqref{eq:pssapp} is a challenging problem and a large amount literature is devoted to this topic. We shall not explore this topic in the current article and our focus is the geomagnetic detection problem. Nevertheless, we need the following technical assumption for our study on the magnetic fields $\mathcal{H}_0$ and $\widetilde{\mathcal{H}}$ in $\mathbb{R}^3\backslash\overline{\Sigma_c}$. Introduce the following temporal Fourier transform for $(\Bx, \omega)\in(\mathbb{R}^3\backslash\overline{\Sigma_c})\times\mathbb{R}_+$,
\begin{equation}\label{eq:cond1}
\mathbf{J}(\Bx,\omega)=\mathcal{F}_t(\mathcal{J}):=\frac{1}{2\pi}\int_0^\infty \mathcal{J}(\Bx, t) e^{\mathrm{i}\omega t}\ dt, \quad\mathcal{J}=\mathcal{H}_0\ \ \mbox{or}\ \ \widetilde{\mathcal{H}}.
\end{equation}
Throughout, we assume that the Fourier transform \eqref{eq:cond1} exists for $(\Bx, \omega)\in(\mathbb{R}^3\backslash\overline{\Sigma_c})\times (0,\omega_0)$ with $\omega_0\in\mathbb{R}_+$ such that $\mathbf{J}(\Bx, \omega)\in H_{loc}(\mathrm{curl},\mathbb{R}^3\backslash\overline{\Sigma_c})$. Here and also in what follows,
\[
H_{loc}(\mathrm{curl}, X):=\{U|_\Omega\in H(\mathrm{curl}, \Omega)|\ \Omega\ \ \mbox{is any bounded subdomain of $X$}\},
\]
and
\[
H(\mathrm{curl}, \Omega)=\{U\in (L^2(\Omega))^3|\ \nabla\times U\in (L^2(\Omega))^3\}.
\]
Set
\begin{equation}\label{eq:ft2}
\mathbf{H}_0=\mathcal{F}_t(\mathcal{H}_0),\ \mathbf{H}=\mathcal{F}_t(\widetilde{\mathcal{H}}).
\end{equation}
From the time-dependent system \eqref{eq:pssapp} as well as the description of the medium configurations in \eqref{eq:gome5}--\eqref{eq:mmm1} associated with $\mathcal{H}_0$ and $\widetilde{\mathcal{H}}$, respectively, one has by direct verifications the following two systems in the frequency-domain, respectively, for $\mathbf{H}_0\in H_{loc}(\mathrm{curl},\mathbb{R}^3\backslash\overline{\Sigma_c})$ and $\mathbf{H}\in H_{loc}(\mathrm{curl},\mathbb{R}^3\backslash\overline{\Sigma_c})$:
\begin{equation}\label{eq:frequen01}
\begin{cases}
\nabla\times\nabla\times \mathbf{H}_0(\Bx)-i\omega\lambda_s^{-1}\mathbf{H}_0(\Bx)=0, &\Bx\in\, \Sigma_s,\medskip \\
\nabla\times\nabla\times \mathbf{H}_0(\Bx)=0,   &\Bx\in\, \RR^3\setminus{\overline{\Sigma}},\medskip \\
\nabla\cdot (\mu(\Bx)\mathbf{H}_0(\Bx))=0, &\Bx\in\, \RR^3\setminus\overline{\Sigma_c},\medskip \\
\mathbf{H}_0=\Ocal(\|\Bx\|^{-2}), & \|\Bx\|\rightarrow \infty,
\end{cases}
\end{equation}
and
\begin{equation}\label{eq:frequen02}
\begin{cases}
\nabla\times\nabla\times \mathbf{H}(\Bx)-i\omega\lambda_s^{-1}\mathbf{H}(\Bx)=0, &\Bx\in\, \Sigma_s\setminus\overline{\bigcup_{l=1}^{l_0}D_l},\medskip \\
\nabla\times\nabla\times \mathbf{H}(\Bx)-i\omega\lambda_l^{-1}\mathbf{H}(\Bx)=0, &\Bx\in\, D_l,\ \ l=1,2,\ldots, l_0,\medskip \\
\nabla\times\nabla\times \mathbf{H}(\Bx)=0,   &\Bx\in\, \RR^3\setminus{\overline{\Sigma}},\medskip \\
\nabla\cdot (\tilde\mu(\Bx)\mathbf{H}(\Bx))=0, &\Bx\in\, \RR^3\setminus\overline{\Sigma_c},\medskip \\
\mathbf{H}=\Ocal(\|\Bx\|^{-2}), & \|\Bx\|\rightarrow \infty.
\end{cases}
\end{equation}
Based on the above Fourier reformulation, the inverse problem \eqnref{eq:geom6} can be recast as
\begin{equation}\label{eq:geom61}
\left(\bH(\Bx, \omega)-\bH_0(\Bx, \omega) \right)\bigg|_{(\Bx,\omega)\in\Gamma\times(0,\omega_0)}\longrightarrow \bigcup_{l=1}^{l_0} (D_l;\mu_l,\sigma_l).
\end{equation}

We shall follow the general strategy developed in \cite{DLL:18} to tackle the inverse problem \eqref{eq:geom61}. That is, we mainly make use of the static parts of the magnetic fields for the recovery of the inverse problem \eqref{eq:geom61}. To that end, one needs to derive the low-frequency asymptotics of the magnetic fields $\mathbf{H}_0$ and $\mathbf{H}$. However, there are several new challenges as well as the corresponding novel technical developments in the current setup of study. The geomagnetic model is a linear Maxwell system in \cite{DLL:18}, and hence the low-frequency asymptotics therein were conducted globally for the whole system. However, the MHD system is a coupled nonlinear system and similar global low-frequency asymptotics of the Maxwell system \eqref{eq:frequen02} in coupling with the whole MHD system shall be fraught with various difficulties. A critical new technical ingredient in the present work is that we treat the measurement data of $\mathbf{H}$ on $\Gamma$ as source inputs for the Maxwell system \eqref{eq:frequen02}, and one actually can continue the data up to the boundary of the core, namely $\partial\Sigma_c$. Indeed, by analytic continuation, one can first analytically extend the data on $\Gamma$ to the closed surface $\widetilde{\Gamma}$, and then by solving the homogeneous Maxwell system in the outer space, one can further extend the data to $\mathbb{R}^3\backslash\overline{\Sigma}$, in particular to the boundary of the Earth $\partial\Sigma$. Since $\lambda_s$ is known in $\Sigma_s$, one can then continue the Cauchy data on $\partial\Sigma$ to $\partial\Sigma_o$ through the Maxwell system in $\Sigma_s\setminus\overline{\bigcup_{l=1}^{l_0}D_l}$. By incorporating such a critical data continuation into our argument, we can then perform the low-frequency asymptotics locally, that is, outside the core of the Earth. We believe this kind of strategy could be potentially used to deal with other inverse problems associated with certain coupled PDEs in different contexts. Furthermore, we would like to point out that in addition to this ``localization" issue, the Maxwell system \eqref{eq:frequen02} is of a different mathematical nature from that considered in \cite{DLL:18}. After deriving the low-frequency asymptotics, we then further linearize the nonlinear inverse problem \eqref{eq:geom61} with respect to the small-size of the magnetized anomalies and finally arrive at the unique recovery results.

The rest of the paper is organized as follows. In Sections 2 and 3, we derive the low-asymptotic expansions of the magnetic fields as well as the leading-order terms of the static magnetic fields with respect to the small-size of the magnetized anomalies. In Section 4, we establish the unique recovery results for the geomagnetic detection problem.

\section{Integral representation and asymptotic analysis }

We make essential use of the layer potential theory in our study. In what follows, we first collect some preliminary knowledge on the layer potential theory. Then we conduct extensive asymptotic analysis in deriving the low-frequency asymptotic expansions of the magnetic fields as well as the leading-order terms of the static magnetic fields with respect to the small-size of the magnetized anomalies associated with the Maxwell systems \eqref{eq:frequen01} and \eqref{eq:frequen02}.

\subsection{Layer potentials}

For a bounded $C^2$ domain $B\subset\mathbb{R}^3$, we first introduce some function spaces on the boundary $\partial B$ for the subsequent use.
Let  $\nabla_{\p B}\cdot$ denote the surface divergence. Denote by $L_T^2(\p B):=\{\Phi\in {L^2(\p B)}^3, \nu\cdot \Phi=0\}$. Let $H^s(\partial B)$ be the usual Sobolev space of order $s\in\mathbb{R}$ on $\partial B$. Set
\begin{align*}
\mathrm{TH}({\rm div}, \p B):&=\Bigr\{ {\Phi} \in L_T^2(\partial B):
\nabla_{\partial B}\cdot {\Phi} \in L^2(\partial B) \Bigr\}, \\
\mathrm{TH}({\rm curl}, \p B):&=\Bigr\{ {\Phi} \in L_T^2(\partial B):
\nabla_{\partial B}\cdot ({\Phi}\times {\nu}) \in L^2(\partial B) \Bigr\},
\end{align*}
endowed with the norms
\begin{align*}
&\|{\Phi}\|_{\mathrm{TH}({\rm div}, \p B)}=\|{\Phi}\|_{L^2(\p B)}+\|\nabla_{\p B}\cdot {\Phi}\|_{L^2(\p B)}, \\
&\|{\Phi}\|_{\mathrm{TH}({\rm curl}, \p B)}=\|{\Phi}\|_{L^2(\p B)}+\|\nabla_{\p B}\cdot({\Phi}\times \nu)\|_{L^2(\p B)},
\end{align*}
respectively.
Let $ \Gamma_k$ be the fundamental solution to the PDO $(\Delta+k^2)$, which is given by
\begin{equation}\label{Gk} \ds \Gamma_k
(\Bx) = -\frac{e^{ik\|\Bx\|}}{4 \pi \|\Bx\|},\ \ \Bx\in\mathbb{R}^3\ \ \mbox{and}\ \ \Bx\neq \mathbf{0}.
 \end{equation}
For any bounded domain $B\subset \RR^3$, we denote by $\Scal_B^k: H^{-1/2}(\p B)\rightarrow H^{1}(\RR^3\setminus\p B)$ the single layer potential operator given by
\beq\label{eq:layperpt1}
\Scal_{B}^{k}[\phi](\Bx):=\int_{\p B}\Gamma_k(\Bx-\By)\phi(\By)d s_\By,
\eeq
and $\Kcal_B^{k}: H^{1/2}(\p B)\rightarrow H^{1/2}(\p B)$ the Neumann-Poincar\'e operator
\beq\label{eq:layperpt2}
\Kcal_{B}^{k}[\phi](\Bx):=\mbox{p.v.}\quad\int_{\p B}\frac{\p\Gamma_k(\Bx-\By)}{\p \nu_y}\phi(\By)d s_\By,
\eeq
where p.v. stands for the Cauchy principle value. In \eqref{eq:layperpt2} and also in what follows, unless otherwise specified, $\nu$ signifies the exterior unit normal vector to the boundary of the concerned domain.
It is known that the single layer potential operator $\Scal_B^k$ satisfies the following trace formula

\beq\label{eq:trace}
\frac{\p}{\p\nu}\Scal_B^{k}[\phi] \Big|_{\pm} = (\pm \frac{1}{2}I+
(\Kcal_{B}^{k})^*)[\phi] \quad \mbox{on} \, \p B, \eeq
where $(\Kcal_{B}^{k})^*$ is the adjoint operator of $\Kcal_B^{k}$.

For a density function $\Phi \in \mathrm{TH}(\mbox{div}, \p B)$, we define the
vectorial single layer potential by
\beq\label{defA}
\ds\mathcal{A}_B^{k}[\Phi](\Bx) := \int_{\p B} \Gamma_{k}(\Bx-\By)
\Phi(\By) d s_\By, \quad \Bx \in \RR^3\setminus\p B.
\eeq
It is known that $\nabla\times\mathcal{A}_B^{k}$ satisfies the following jump formula
\begin{equation}\label{jumpM}
\nu \times \nabla \times \mathcal{A}_B^{k}[\Phi]\big\vert_\pm = \mp \frac{\Phi}{2} + \mathcal{M}_B^{k}[\Phi] \quad \mbox{on}\,  \p B,
\end{equation}
where \begin{equation*}
\forall \Bx\in \p B, \quad \nu \times \nabla \times \mathcal{A}_B^{k}[\Phi]\big\vert_\pm (\Bx)= \lim_{t\rightarrow 0^+} \nu \times \nabla \times \mathcal{A}_B^{k}[\Phi] (\Bx\pm t \nu)
\end{equation*}
and
\beq\label{Mk}
\mathcal{M}^{k}_B[\Phi](\Bx)= \mbox{p.v.}\quad\nu  \times \nabla \times \int_{\p B} \Gamma_{k}(\Bx-\By) \Phi(\By) d s_\By.
\eeq
We also define $\mathcal{L}^k_B: L^2(\p B) \rightarrow \mathrm{TH}({\rm div}, \p B)$ by
\beq\label{Lk}
\mathcal{L}^k_B[\varphi](\Bx):= \nu_\Bx  \times \nabla\mathcal{S}_B^k[\varphi](\Bx),
\eeq
and $\mathcal{N}^k_B: \mathrm{TH}({\rm div}, \p B) \rightarrow L^2(\p B)$ by
\beq\label{Nk}
\mathcal{N}^k_B[\Phi](\Bx):= \nu_\Bx  \cdot \Big(\nabla\times \mathcal{A}_B^k[\Phi](\Bx)\Big).
\eeq
{It is mentioned that $\frac{I}{2}\pm \Mcal_{B}^{k}$ is invertible on $\rm{TH}({\rm div}, \p B)$} when $k$ is sufficiently small (see e.g., \cite{ADM14,T}). In the following, if $k=0$, we formally set $\Gamma_{k}$ introduced in \eqref{Gk} to be $\Gamma_0=-1/(4\pi\|\Bx\|)$, and the other integral operators introduced above can also be formally defined when $k=0$.

Throughout the rest of the paper, we set
\[
k^{2}:=-i\frac{\omega}{\lambda},\ \ \Re k\geq0,
\]
where $\lambda$ is a constant magnetic diffusivity. $k_s^2$, $k_c^2$ and $k_l^2$, $l=1, 2, \ldots, l_0$ are defined similarly to $k^2$ by replacing $\lambda$ with $\lambda_s$, $\lambda_c$ and $\lambda_l$, $l=1, 2, \ldots, l_0$, respectively.

\subsection{Low-frequency asymptotic expansions of the magnetic fields}

In this subsection, we give the low-frequency asymptotic expansions of the magnetic fields $\mathbf{H}_0$ and $\mathbf{H}$ associated with \eqref{eq:frequen01} and \eqref{eq:frequen02}, respectively. We would like to mention in passing some related literature on the low-frequency asymptotic analysis of the Maxwell system \cite{A1,A2,A3,Das02,DHH:17,Kle}.
We first deal with $\mathbf{H}_0$, which can be represented by the following integral ansatz,
\beq\label{eq:repre02}
\mathbf{H}_0=\left\{\begin{split}
&\nabla\times \Acal^{0}_{\Sigma_c}[\Phi_c]+\nabla\times \Acal^{0}_{\Sigma}[\Phi_s]
+ \nabla\times\nabla\times\Acal^{0}_{\Sigma}[\nabla_{\p \Sigma}\varphi_s]\ \ \ \  \mbox{in} \, \ \RR^3\setminus{\overline{\Sigma}},\medskip\\
&\nabla\times \Acal^{k_s}_{\Sigma_c}[\Phi_c]+\nabla\times \Acal^{k_s}_{\Sigma}[\Phi_s]
+ \nabla\times\nabla\times\Acal^{k_s}_{\Sigma}[\nabla_{\p \Sigma}\varphi_s]\, \, \, \ \mbox{in} \, \ \Sigma_s,
\end{split}
\right.
\eeq
where $(\Phi_c, \Phi_s, \varphi_s)\in\mathrm{TH}(\rm div,\partial\Sigma_c)\times\mathrm{TH}(\rm div,\partial\Sigma)\times {\rm H}^2(\partial\Sigma).$
By the continuous property of  $\nu\times\mathbf{H}_0$ across the boundary $\partial\Sigma_c$ and $\partial\Sigma$, one has
\beq\label{eq:asymtmp01}
\begin{split}
&\nu\times\nabla\times \Acal^{k_s}_{\Sigma_c}[\Phi_c]|_{\partial\Sigma_c}^{+}+\nu\times\nabla\times \Acal^{k_s}_{\Sigma}[\Phi_s]|_{\partial\Sigma_c}^{+}+\nu\times\nabla\times\nabla\times\Acal^{k_s}_{\Sigma}[\nabla_{\p \Sigma}\varphi_s]|_{\partial\Sigma_c}^{+}=\nu\times\mathbf{H}_0|_{\partial\Sigma_c}^{+},
\end{split}
\eeq
and
\beq\label{eq:asymtmp02}
\begin{split}
&\nu\times\nabla\times \Acal^{k_s}_{\Sigma_c}[\Phi_c]|_{\partial\Sigma}^{-}+
\nu\times\nabla\times \Acal^{k_s}_{\Sigma}[\Phi_s]|_{\partial\Sigma}^{-}+\nu\times\nabla\times\nabla\times\Acal^{k_s}_{\Sigma}[\nabla_{\p \Sigma}\varphi_s]|_{\partial\Sigma}^{-}\medskip\\
=&\nu\times\nabla\times \Acal^0_{\Sigma_c}[\Phi_c]|_{\partial\Sigma}^{+}+\nu\times\nabla\times \Acal^0_{\Sigma}[\Phi_s]|_{\partial\Sigma}^{+}+ \nu\times\nabla\times\nabla\times\Acal^{0}_{\Sigma}[\nabla_{\p \Sigma}\varphi_s]|_{\partial\Sigma}^{+},
\end{split}
\eeq
On the other hand, by the continuous of $\nu\cdot \mu\bH_0$ ((see Lemma 3.3 in \cite{DHU:17})) across the boundary $\partial\Sigma$, one further has
\beq\label{eq:asymtmp04}
\begin{split}
&\nu\cdot\Big(\nabla\times \Acal^{k_s}_{\Sigma_c}[\Phi_c]\Big)|_{\partial\Sigma}^{-}+\nu\cdot\Big(\nabla\times \Acal^{k_s}_{\Sigma}[\Phi_s]\Big)|_{\partial\Sigma}^{-}+\nu\cdot\nabla\times\nabla\times\Acal^{k_s}_{\Sigma}[\nabla_{\p \Sigma}\varphi_s]|_{\partial\Sigma}^{-}\medskip\\
=&\nu\cdot\Big(\nabla\times \Acal^0_{\Sigma_c}[\Phi_c]\Big)|_{\partial\Sigma}^{+}+\nu\cdot\Big(\nabla\times \Acal^0_{\Sigma}[\Phi_s]\Big)|_{\partial\Sigma}^{+}+\nu\cdot\nabla\times\nabla\times\Acal^{0}_{\Sigma}[\nabla_{\p \Sigma}\varphi_s]|_{\partial\Sigma}^{+}.
\end{split}
\eeq
Define $\Mcal_{\Sigma_c, \Sigma}^{k'}:=\nu\times\nabla\times \Acal^{k'}_{\Sigma}|_{\partial\Sigma_c},\ \Mcal_{\Sigma, \Sigma_c}^{k'}:=\nu\times\nabla\times \Acal^{k'}_{\Sigma_c}|_{\partial\Sigma},$ and
$\Lcal_{\Sigma_c, \Sigma}^{k'}:=\nu\times\nabla\times\nabla\times\Acal^{k'}_{\Sigma}|_{\partial\Sigma_c},\ \Lcal_{\Sigma, \Sigma_c}^{k'}:=\nu\times\nabla\times\nabla\times\Acal^{k'}_{\Sigma_c}|_{\partial\Sigma},\ k'=0, k_s.$
By using \eqnref{eq:asymtmp01}, \eqnref{eq:asymtmp02} and the jump formula \eqnref{jumpM}, there holds
\beq\label{eq:H00}
\Big(-\frac{I}{2}+\Mcal^{k_s}_{\Sigma_c}\Big)[\Phi_c]+\Mcal_{\Sigma_c, \Sigma}^{k_s}[\Phi_s]+\Lcal_{\Sigma_c, \Sigma}^{k_s}[\nabla_{\p \Sigma}\varphi_s]=\nu\times\mathbf{H}_0|_{\partial\Sigma_c}^{+},
\eeq
and
\beq\label{eq:MM0}
\Phi_s+\big(\Mcal^{k_s}_{\Sigma}-\Mcal^0_{\Sigma}\big)[\Phi_s]+\big(\Mcal_{\Sigma, \Sigma_c}^{k_s}-\Mcal_{\Sigma,\Sigma_c}^0\big)[\Phi_c]+\big(\Lcal^{k_s}_{\Sigma}-\Lcal^0_{\Sigma}\big)[\nabla_{\p \Sigma}\varphi_s]=0.
\eeq
Define $\Ncal_{\Sigma_c, \Sigma}^{k'}:=\nu\cdot\Big(\nabla\times \Acal^{k'}_{\Sigma}|_{\partial\Sigma_c}\Big),\ \Ncal_{\Sigma, \Sigma_c}^{k'}:=\nu\cdot\Big(\nabla\times \Acal^{k'}_{\Sigma_c}|_{\partial\Sigma}\Big),\ k'=0, k_s.$ By using \eqnref{eq:asymtmp04}, noting that $\nabla\times\nabla\times=-\Delta+\nabla\nabla\cdot$, and the jump formula \eqnref{eq:trace}, one further has
\beq\label{eq:MM1}
\begin{split}
&(\Ncal_{\Sigma,\Sigma_c}^{k_s}-\Ncal_{\Sigma,\Sigma_c}^{0})[\Phi_c]+(\Ncal_{\Sigma}^{k_s}-\Ncal_{\Sigma}^0)[\Phi_s]
+k_s^2\nu\cdot\Acal_{\Sigma}^{k_s}[\nabla_{\p \Sigma}\varphi_s]\\
&-\nabla_{\p \Sigma}\cdot\nabla_{\p \Sigma}\varphi_s+\Big((\Kcal^{k_s}_{\Sigma})^*-(\Kcal^{0}_{\Sigma})^*\Big)[\nabla_{\p \Sigma}\cdot\nabla_{\p \Sigma}\varphi_s]=0.
\end{split}
\eeq
We shall need the following elementary estimates (see \cite{ADM14,DLL:18}),
\begin{lem}\label{le:basicasym01}
There hold the following asymptotic relationships for $k_s$ sufficiently small:
\beq\label{eq:M0}
\begin{split}
&\Mcal_{\Sigma', \Sigma''}^{k_s}=\Mcal_{\Sigma', \Sigma''}^{0}+\Ocal(k_s^2), \quad \Lcal^{k_s}_{\Sigma', \Sigma''}=\Lcal^{0}_{\Sigma', \Sigma''}+\Ocal(k_s^2),\\
&\Ncal^{k_s}_{\Sigma', \Sigma''}=\Ncal^{0}_{\Sigma', \Sigma''}+\Ocal(k_s^2), \quad (\Kcal^{k_s}_{\Sigma'})^*=(\Kcal^{0}_{\Sigma'})^*+\Ocal(k_s^2),
\end{split}
\eeq
where $\Sigma', \Sigma''\in \{\Sigma, \Sigma_c, D_1, D_2, \ldots, D_l\}$. Furthermore, if $\Sigma'=\Sigma''$ then the subscripts of the related operators in \eqref{eq:M0} reduces to $\Sigma'$ only.
\end{lem}
Based on \eqnref{eq:H00}-\eqnref{eq:MM1} and Lemma \ref{le:basicasym01}, we can derive the following asymptotic expansion of $\bH_0$, whose proof is given in Appendix A.
\begin{lem}\label{le:asymH001}
Let $\bH_0$ be the solution to \eqnref{eq:frequen01}. Then for $\omega\in\mathbb{R}_+$ sufficiently small, there holds the following asymptotic expansion:
\beq\label{eq:le01sol01}
\bH_0=\nabla\times \Acal^{0}_{\Sigma_c}\Big(-\frac{I}{2}+\Mcal_{\Sigma_c}^0\Big)^{-1}[\nu\times\bH_0|_{\p \Sigma_c}^+]+\Ocal(k_s^2) \quad \mbox{in} \quad \RR^3\setminus\overline{\Sigma_c}.
\eeq
\end{lem}

Next we establish the asymptotic expansion of the magnetic field $\mathbf{H}$ associated with the system \eqnref{eq:frequen02}.  First, by using standard layer potential theory, one can directly verify that $\mathbf{H}$ can be given by the following integral ansatz,
\beq\label{eq:repre02}
\bH=\left\{\begin{split}
&\nabla\times\Acal_{\Sigma_c}^0[\Psi_0]+\nabla\times\Acal_{\Sigma}^0[\Phi_0]+ \nabla\times\nabla\times\Acal_{\Sigma}^0[\nabla_{\p \Sigma}\varphi_0]\\
&\quad+\nabla\times\sum_{l'=1}^{l_0}\Acal_{D_{l'}}^{0}[\Phi_{l'}]+ \nabla\times\nabla\times\sum_{l'=1}^{l_0}\Acal_{D_{l'}}^{0}[\nabla_{\p D_{l'}}\varphi_{l'}] \quad \mbox{in} \ \  \RR^3\setminus\overline{\Sigma},\\
&\nabla\times\Acal_{\Sigma_c}^{k_s}[\Psi_0]+\nabla\times\Acal_{\Sigma}^{k_s}[\Phi_0]+ \nabla\times\nabla\times\Acal_{\Sigma}^{k_s}[\nabla_{\p \Sigma}\varphi_0]\\
&\quad+\nabla\times\sum_{l'=1}^{l_0}\Acal_{D_{l'}}^{k_s}[\Phi_{l'}]+ \nabla\times\nabla\times\sum_{l'=1}^{l_0}\Acal_{D_{l'}}^{k_s}[\nabla_{\p D_{l'}}\varphi_{l'}] \quad \mbox{in}\ \ \Sigma_s\setminus\overline{\bigcup_{l'=1}^{l_0} D_{l'}}, \\
&\nabla\times\Acal_{\Sigma_c}^{k_l}[\Psi_0]+\nabla\times\Acal_{\Sigma}^{k_l}[\Phi_0]+ \nabla\times\nabla\times\Acal_{\Sigma}^{k_l}[\nabla_{\p \Sigma}\varphi_0]\\
&\quad+\nabla\times\sum_{l'=1}^{l_0}\Acal_{D_{l'}}^{k_{l}}[\Phi_{l'}]+ \nabla\times\nabla\times\sum_{l'=1}^{l_0}\Acal_{D_{l'}}^{k_l}[\nabla_{\p D_{l'}}\varphi_{l'}] \quad \mbox{in} \ \ D_{l},\, l=1, 2, \cdots, l_0,
\end{split}
\right.
\eeq
where the densities $(\Psi_0, \Phi_0, \Phi_1, \Phi_2, \ldots, \Phi_{l_0})$ satisfy
$(\Psi_0, \Phi_0, \Phi_1, \Phi_2, \ldots, \Phi_{l_0})\in {\rm TH}({\rm div}, \p \Sigma_c)\otimes {\rm TH}({\rm div}, \p \Sigma)\otimes{\rm TH}({\rm div}, \p D_1)\otimes{\rm TH}({\rm div}, \p D_2)\otimes \cdots\otimes{\rm TH}({\rm div}, \p D_{l_0}).$ The densities
$(\varphi_0, \varphi_1,\ldots, \varphi_{l_0})$ belongs to $H^1(\p \Sigma)\otimes H^1(\p D_1)\otimes
\cdots\otimes H^1(\p D_{l_0})$. We show the uniqueness of the densities in \eqnref{eq:repre02} when $\omega$ is sufficiently small. Based on the above integral representation, we can show the following asymptotic expansion of the magnetic field $\bH$, where proof is given in Appendix B.
\begin{lem}\label{le:asymH01}
Let $\bH$ be the solution to \eqnref{eq:frequen02}. Then for $\omega\in\mathbb{R}_+$ sufficiently small, there holds the following asymptotic expansion:
\beq\label{eq:le02sol01}
\bH=\nabla\times\Acal_{\Sigma_c}^0[\Psi_0^{(0)}]+\nabla\sum_{l=1}^{l_0}\Scal_{D_{l}}^{0}[\varphi_{l}^{(0)}]+\Ocal(\omega) \quad \mbox{in} \quad \RR^3\setminus\overline{\Sigma_c},
\eeq
where $(\varphi_1^{(0)}, \varphi_2^{(0)}, \ldots, \varphi_{l_0}^{(0)})\in L^2(\p D_1)\otimes L^2(\p D_2)\otimes \cdots\otimes L^2(\p D_{l_0})$ are solutions to the following integral equations:
\beq\label{eq:le02sol02}
\begin{split}
&\left(\varsigma_l I -(\Kcal_{D_l}^0)^*+\Pcal_{D_l, \Sigma_c}\Lcal_{\Sigma_c, D_l}^0\right)[\varphi_l^{(0)}]-\sum_{l'\neq l}^{l_0} \Big(\Kcal_{D_l, D_{l'}}^0-\Pcal_{D_l, \Sigma_c}\Lcal_{\Sigma_c, D_{l'}}^0\Big)[\varphi_{l'}^{(0)}] \\
=&\Pcal_{D_l, \Sigma_c}[\nu\times\bH^0|_{\p \Sigma_c}] \quad \mbox{on} \quad \p D_l, \ l=1,2,\ldots,l_0,
\end{split}
\eeq
where
\beq\label{eq:defsigmal}
\varsigma_l:=\frac{\mu_l+\mu_0}{2(\mu_l-\mu_0)},
\eeq
and the operators $\Kcal_{D_l, D_{l'}}^0$ are defined in \eqnref{eq:defKcalsp01} and $\Pcal_{D_l, \Sigma_c}$ are defined in \eqnref{eq:defPcal01}, $l, l'=1, 2, \ldots, l_0$. Furthermore, $\Psi_0^{(0)}\in {\rm TH}({\rm div}, \p \Sigma_c)$ is the solution to the following integral equation:
\beq\label{eq:le02sol03}
\Psi_0^{(0)}=\Big(-\frac{I}{2}+\Mcal_{\Sigma_c}^0\Big)^{-1}\Big[\nu\times\bH^0|_{\p \Sigma_c}\Big]-\sum_{l'=1}^{l_0}\Big(-\frac{I}{2}+\Mcal_{\Sigma_c}^0\Big)^{-1}\Lcal_{\Sigma_c, D_{l'}}^{0}[\varphi_{l'}^{(0)}].
\eeq
Here, we denote by $\bH^0$ the leading-order term of $\bH$ with respect to $\omega$, i.e., $\bH=\bH^0+\Ocal(\omega)$.
\end{lem}

\subsection{Asymptotic expansions of the static magnetic fields with respect to the anomaly size}\label{sec:02}

In this section, we make further asymptotic expansions of the static magnetic fields derived in the previous subsection with respect to the size of the magnetized anomalies. In fact, from a practical point of view, the size of the magnetized anomalies $(D_l; \mu_l, \sigma_l)$, $l=1,2,\ldots,l_0$, introduced in \eqref{eq:paradefsigma}, is much smaller than the size of the Earth. Hence, we can assume that
\beq\label{eq:permeab02}
D_l=\delta\Omega+\bZ_l, \quad l=1, 2, \ldots, l_0,
\eeq
where $\Omega$ is a bounded Lipschitz domain in $\RR^3$ with $\Omega\Subset \Sigma$, and $\delta\in\mathbb{R}_+$ is sufficiently small.
Furthermore, we assume that $D_l$, $l=1, 2, \ldots, l_0$ are sparsely distributed and $\Bz_l$, $l=1, 2, \ldots, l_0$, are sufficiently far away from $\p \Sigma_s$ such that $\Bx-\Bz_l\gg\delta$, for any $\Bx\in\p \Sigma_s$. With the above preparations, we are in a position to derive the asymptotic expansions of the static geomagnetic fields with respect to the size of the magnetized anomalies. To that end, we define a matrix operator $\mathbb{M}_{D}$ on $L^2(\p D_1)\otimes L^2(\p D_2)\otimes\ldots\otimes L^2(\p D_{l_0})$ by
\beq\label{eq:asymtmpMD}
\mathbb{M}_{D}:=\left[
\begin{array}{cccc}
\Qcal_{D_1} & \Qcal_{D_1,D_2} &\cdots &\Qcal_{D_1,D_{l_0}}\\
\Qcal_{D_2,D_1} & \Qcal_{D_2} &\cdots &\Qcal_{D_2,D_{l_0}}\\
\vdots & \vdots &\ddots  \vdots \\
\Qcal_{D_{l_0},D_1} & \Qcal_{D_{l_0}, D_2} &\cdots &\Qcal_{D_{l_0}}
\end{array}
\right],
\eeq
where $\Qcal_{D_l, D_{l'}}:=\Kcal_{D_l, D_{l'}}^0-\Pcal_{D_l, \Sigma_c}\Lcal_{\Sigma_c, D_{l'}}^0$ and $\Kcal_{D_l, D_{l'}}^0$ are defined in \eqnref{eq:defKcalsp01}, $l, l'= 1, 2, \ldots, l_0$.
%\beq\label{eq:asymtmpMD}
%\mathbb{M}_{D}:=\left[
%\begin{array}{llll}
%(\Kcal_{D_1}^0)^*-\Pcal_{D_1, \Sigma_c}\Lcal_{\Sigma_c, D_1}^0 & \Kcal_{D_1,D_2}^0-\Pcal_{D_1, \Sigma_c}\Lcal_{\Sigma_c, D_2}^0 &\cdots &\Kcal_{D_1,D_{l_0}}^0-\Pcal_{D_1, \Sigma_c}\Lcal_{\Sigma_c, D_{l_0}}^0\\
%\Kcal_{D_2,D_1}^0-\Pcal_{D_2, \Sigma_c}\Lcal_{\Sigma_c, D_1}^0 & (\Kcal_{D_2}^{0})^*-\Pcal_{D_2, \Sigma_c}\Lcal_{\Sigma_c, D_2}^0 &\cdots &\Kcal_{D_2,D_{l_0}}^0-\Pcal_{D_2, \Sigma_c}\Lcal_{\Sigma_c, D_{l_0}}^0\\
%\vdots & \vdots &\ddots  \vdots \\
%\Kcal_{D_{l_0},D_1}^0-\Pcal_{D_{l_0}, \Sigma_c}\Lcal_{\Sigma_c, D_1}^0 & \Kcal_{D_{l_0}, D_2}^{0}-\Pcal_{D_{l_0}, \Sigma_c}\Lcal_{\Sigma_c, D_{l_0}}^0 &\cdots &(\Kcal_{D_{l_0}}^0)^*-\Pcal_{D_{l_0}, \Sigma_c}\Lcal_{\Sigma_c, D_{l_0}}^0
%\end{array}
%\right]
%\eeq

We first have the following lemma
\begin{lem}\label{le:smallobj01}
Suppose $D_l$, $l=1, 2 ,\ldots, l_0$ are defined in \eqnref{eq:permeab02} with $\delta\in\mathbb{R}_+$ sufficiently small. Let $\mathbb{M}_{D}$ be defined in \eqnref{eq:asymtmpMD}. If $(\phi_0, \phi_1, \phi_2, \ldots, \phi_{l_0})\in L_0^2(\p D_1)\otimes L_0^2(\p D_2)\otimes\ldots\otimes L_0^2(\p D_{l_0})$, where $L_0^2(\p D_l)$, $1\leq l\leq l_0$ is the $L^2(\p D_l)$ space whose elements have zero means on the boundary, then there holds
\beq\label{eq:MDd}
\mathbb{M}_{D}(\phi_0, \phi_1, \phi_2, \ldots, \phi_{l_0})^T=\mathbb{N}_{\Omega}(\tilde\phi_0, \tilde\phi_1, \tilde\phi_2, \ldots, \tilde\phi_{l_0})^T +\Ocal(\delta^2),
\eeq
where $\mathbb{N}_{\Omega}$ is an $l_0\times l_0$ matrix-valued operator defined by
\beq\label{eq:lesmobj0102}
\mathbb{N}_\Omega:={\rm diag}((\Kcal_{\Omega}^0)^*,(\Kcal_{\Omega}^0)^*,\ldots,(\Kcal_{\Omega}^0)^*).
\eeq
Here, the density functions $\tilde\phi_l$, $l=0,1,\ldots,l_0$, are defined as, $\tilde\phi_{l}(\tdy):=\phi_{l}(\By)$, where $\tdy=\delta^{-1}(\By-\Bz_{l}) \in \p  \Omega$, $l\in \{1, 2, \ldots, l_0\}$.
\end{lem}

\begin{proof}
For any $\Bx\in \p D_l$, $\By\in \p D_{l'}$ and $\tdx=\delta^{-1}(\Bx-\Bz_l)$, $\tdy=\delta^{-1}(\By-\Bz_{l'}) \in \p  \Omega$, where $l, l'\in \{1, 2, \ldots, l_0\}$, one can show that
\beq\label{eq:MDlm}
\begin{split}
\Pcal_{D_l, \Sigma_c}\Lcal_{\Sigma_c, D_{l'}}^0[\phi](\Bx)=&\nu_\Bx  \cdot  \nabla_\Bx \times \Acal_{\Sigma_c}^0\Big(-\frac{I}{2}+\Mcal_{\Sigma_c}^0\Big)^{-1}\nu \times \nabla \int_{\p D_{l'}} \Gamma_0(\cdot-\By)\phi(\By)d s_\By \\
=&\delta\nu_{\tdx}  \cdot  \nabla_{\tdx}\times \Acal_{\Sigma_c}^0\Big(-\frac{I}{2}+\Mcal_{\Sigma_c}^0\Big)^{-1}\nu \times \nabla \int_{\p \Omega} \Gamma_0(\cdot-\delta\tdy-\Bz_l)\tilde{\phi}(\tdy)d s_{\tdy}\\
=&\delta\nu_{\tdx}  \cdot  \nabla_{\tdx}\times \Acal_{\Sigma_c}^0\Big(-\frac{I}{2}+\Mcal_{\Sigma_c}^0\Big)^{-1}\nu \times \nabla \Gamma_0(\cdot-\Bz_l)\int_{\p \Omega} \tilde{\phi}_{l'}(\tdy)d s_{\tdy}+\Ocal(\delta^2)\\
=& \Ocal(\delta^2).
\end{split}
\eeq
Similarly one can show that
\beq\label{eq:KDlm}
\Kcal_{D_l, D_{l'}}^0[\phi]=\Ocal(\delta^2), \quad l\neq l', \quad \mbox{and} \quad (\Kcal_{D_{l'}}^0)^*[\phi]=(\Kcal_{\Omega}^0)^*[\tilde\phi]+\Ocal(\delta^2).
\eeq
By substituting \eqnref{eq:MDlm} and \eqnref{eq:KDlm} back into \eqnref{eq:asymtmpMD}, one readily has \eqnref{eq:MDd}.

The proof is complete.
\end{proof}
In what follows, we let $\bH_0^0$ and $\bH^0$ respectively be the leading-order terms of $\bH_0$ and $\bH$ with respect to the low-frequency $\omega$; see \eqref{eq:le01sol01} and \eqref{eq:le02sol01}. By the asymptotic estimate in Lemma \ref{le:smallobj01}, one can derive the following asymptotic property of the densities $\varphi_1^{(0)}, \varphi_2^{(0)}, \ldots, \varphi_{l_0}^{(0)}$ and $\Psi_0^{(0)}$ defined in Lemma \ref{le:asymH01}.
\begin{lem}
Let $D_l$, $l=1, 2, \ldots, l_0$ be given by \eqnref{eq:permeab02}. Let $\varphi_1^{(0)}, \varphi_2^{(0)}, \ldots, \varphi_{l_0}^{(0)}$ and $\Psi_0^{(0)}$ be defined in Lemma \ref{le:asymH01} and $\tilde\varphi_{l}^{(0)}(\tdy):=\varphi_{l}^{(0)}(\By)$, where $\tdy=\delta^{-1}(\By-\Bz_{l}) \in \p  \Omega$, $l=1, 2, \ldots, l_0$. Then there hold the following estimates:
\beq\label{eq:asymPhi02}
\tilde\varphi_{l}^{(0)}=(\varsigma_l I - (\Kcal_\Omega^0)^*)^{-1}[\nu]\cdot\mathbb{H}(\Bz_l) + \Ocal(\delta),\ \ l=1,2,\ldots,l_0,
\eeq
and
\beq\label{eq:asymPhi01}
\begin{split}
\Psi_0^{(0)}=&\Big(-\frac{I}{2}+\Mcal_{\Sigma_c}^0\Big)^{-1}\Big[\nu\times\bH^0|_{\p \Sigma_c}\Big]+\Ocal(\delta^4)+\\
&\delta^3\Big(-\frac{I}{2}+\Mcal_{\Sigma_c}^0\Big)^{-1}\Big[\nu\times \sum_{l'=1}^{l_0}\nabla^2\Gamma_0(\cdot-\Bz_l)\int_{\p \Omega}\tdy(\varsigma_l I - (\Kcal_\Omega^0)^*)^{-1}[\nu](\tdy)d s_{\tdy}\mathbb{H}(\Bz_l)\Big],
\end{split}
\eeq
where $\varsigma_l$ is defined by \eqnref{eq:defsigmal} and $\mathbb{H}$ is defined by
\beq\label{eq:asymPhi03}
\mathbb{H}:=\nabla\times\Acal_{\Sigma_c}^0\left(-\frac{I}{2}+\Mcal_{\Sigma_c}^{0}\right)^{-1}\Big[\nu\times\bH^0|_{\p \Sigma_c}\Big].
\eeq
\end{lem}
\begin{proof}
First, by using \eqnref{eq:asymtmpMD} one can rewrite \eqnref{eq:le02sol02} in the following form,
\beq\label{eq:le02sol04}
(\varsigma -\mathbb{M}_{D})[\varphi]=\mathbf{H}_c,
\eeq
where $\varsigma:={\rm diag}(\varsigma_1, \varsigma_2, \ldots, \varsigma_{l_0})$ and
$$\mathbf{H}_c:=(\Pcal_{D_1, \Sigma_c}[\nu\times\bH^0|_{\p \Sigma_c}], \Pcal_{D_2, \Sigma_c}[\nu\times\bH^0|_{\p \Sigma_c}], \ldots, \Pcal_{D_{l_0}, \Sigma_c}[\nu\times\bH^0|_{\p \Sigma_c}])^T.$$
We shall analyze the right hand side term in \eqnref{eq:le02sol04}. By definition,
\beq\label{eq:appen0101}
\Pcal_{D_l, \Sigma_c}[\nu\times\bH^0|_{\p \Sigma_c}]=\nu\cdot\nabla\times\Acal_{\Sigma_c}^0\left(-\frac{I}{2}+\Mcal_{\Sigma_c}^{0}\right)^{-1}\Big[\nu\times\bH^0|_{\p \Sigma_c}\Big] \quad \mbox{on} \quad \p D_l.
\eeq
which together with the Taylor series expansion of $\mathbb{H}$ on $\Bz_l$, i.e.,
\beq\label{eq:appen0104}
\mathbb{H}(\By)=\mathbb{H}(\Bz_l)+\nabla \mathbb{H}(\Bz_l)(\By-\Bz_l)+\Ocal(\|\By-\Bz_l\|^2), \quad \By\in \p D_l,
\eeq
readily yields that
\beq\label{eq:appen0105}
\Pcal_{D_l, \Sigma_c}[\nu\times\bH^0|_{\p \Sigma_c}]=\nu\cdot\mathbb{H}(\Bz_l)+ \Ocal(\delta) \quad \mbox{on} \quad \p \Omega.
\eeq
By combing \eqnref{eq:MDd}, \eqnref{eq:lesmobj0102}, \eqnref{eq:le02sol04} and \eqnref{eq:appen0105}, one thus has \eqnref{eq:asymPhi02}.
Next, by taking integration on both sides of \eqnref{eq:le02sol02} over $\p D_l$ and using integration by parts, one can derive that
\beq\label{eq:main01pf02}
\int_{\p D_l}\left(\varsigma_l I -(\Kcal_{D_l}^0)^*+\Pcal_{D_l, \Sigma_c}\Lcal_{\Sigma_c, D_l}^0\right)[\varphi_l^{(0)}]ds=(\varsigma_l-\frac{1}{2})\int_{\p D_l}\varphi_l^{(0)}ds=0,
\eeq
where the well-known result $\Kcal_{D_l}^0[1]=1/2$ is used (see, e.g., \cite{HK07:book}). Note that $(\varsigma_l-\frac{1}{2})\neq 0$, one thus has
\beq\label{eq:main01pf03}
\int_{\p D_l}\varphi_l^{(0)}ds=0, \quad l=1, 2, \ldots, l_0.
\eeq
By using \eqnref{eq:MDlm}, one can further show that for $\Bx\in \p \Sigma_c$,
\beq\label{eq:appen0107}
\begin{split}
\Lcal_{\Sigma_c, D_{l'}}^{0}[\varphi_{l'}^{(0)}](\Bx)=&\nu \times \nabla \int_{\p D_{l'}} \Gamma_0(\Bx-\By)\varphi_{l'}^{(0)}(\By)d s_\By  \\
&=\delta^2\nu \times \nabla \int_{\p \Omega} \Gamma_0(\Bx-\delta\tdy-\Bz_l)\tilde\varphi_{l'}^{(0)}(\tdy)d s_{\tdy}\\
&=-\delta^3\nu\times \nabla^2 \Gamma_0(\Bx-\Bz_l)\int_{\p \Omega}\tdy\tilde\varphi_{l'}^{(0)}(\tdy)d s_{\tdy}+\Ocal(\delta^4).
\end{split}
\eeq
Thus from \eqnref{eq:le02sol03} one immediately proves \eqnref{eq:asymPhi01}, which completes the proof.
\end{proof}

We are in a position to present our first main result.

\begin{thm}\label{th:main01}
Suppose $\bH_0$ and $\bH$ are solutions to \eqnref{eq:frequen01} and \eqnref{eq:frequen02}, respectively. Suppose $D_l$, $l=1, 2, \ldots, l_0$ satisfy \eqnref{eq:permeab02}, and let $\bH_0^0$ and $\bH^0$ be the leading-order terms with respect to $\omega\ll 1$ of $\bH_0$ and $\bH$, respectively. Then for $\Bx\in \RR^3\setminus\overline{\Sigma}$, there holds the following asymptotic expansion:
\beq\label{eq:thmain01n}
\begin{split}
(\bH^0-\bH_0^0)(\Bx)=&\nabla\times\Acal_{\Sigma_c}^0\left(-\frac{I}{2}+\Mcal_{\Sigma_c}^{0}\right)^{-1}\Big[\nu\times(\bH^0-\bH_0^0)|_{\p \Sigma_c}^+\Big](\Bx)\\
&+\delta^3\nabla\Scal_{\Sigma_c}^0\left(-\frac{I}{2}+(\Kcal_{\Sigma_c}^{0})^*\right)^{-1}\Big[\nu\cdot \nabla^2 \sum_{l=1}^{l_0}\Gamma_0(\cdot-\Bz_l)\mathbf{M}_l\mathbb{H}(\Bz_l)\Big](\Bx)\\
&-\delta^3\sum_{l=1}^{l_0}\nabla^2 \Gamma_0(\Bx-\Bz_l)\mathbf{M}_l\mathbb{H}(\Bz_l)+\Ocal(\delta^4),
\end{split}
\eeq
where $\mathbf{M}_l$, $l=1, 2, \ldots, l_0$, are the polarization tensors defined by
\beq\label{eq:thmain02}
\mathbf{M}_l:=\int_{\p \Omega} \By (\varsigma_l I - (\Kcal_\Omega^0)^*)^{-1}[\nu](\By) ds_\By,
\eeq
and $\varsigma_l$ and $\mathbb{H}$ are respectively defined in \eqnref{eq:defsigmal} and \eqnref{eq:asymPhi03}.
\end{thm}

Before presenting the proof of Theorem \ref{th:main01}, we remark that in the formula \eqnref{eq:thmain01n}, the asymptotic expansion of $\bH^0-\bH_0^0$ contains the term $\bH^0-\bH_0^0$ on $\p \Sigma_c$ (in the first term of the right hand side of \eqnref{eq:thmain01n}), which is not directly given. However, according to our earlier discussion about the data continuation, one actually can continue the measurement data on $\Gamma$ to $\Sigma_c$. Nevertheless, in the next section, we shall perform further asymptotic analysis to show that  this data continuation step can actually be relaxed.

\begin{proof}[Proof of Theorem~\ref{th:main01}]
By using Taylor's series expansion, \eqnref{eq:asymPhi02} and \eqnref{eq:main01pf03}, one can show that
\beq\label{eq:main01pf04}
\begin{split}
\nabla\sum_{l'=1}^{l_0}\Scal_{D_{l'}}^{0}[\varphi_{l'}^{(0)}](\Bx)=&\sum_{l'=1}^{l_0}\sum_{|\alpha|=0}^{\infty}\frac{(-1)^{|\alpha|}}{\alpha!}\nabla\int_{\p D_l'}\p^{\alpha}\Gamma_0(\Bx-\Bz_l)(\By-\Bz_l)^{\alpha}\varphi_{l'}^{(0)}(\By)ds_\By\\
=&-\delta^3\sum_{l'=1}^{l_0}\nabla^2\Gamma_0(\Bx-\Bz_l)\int_{\p \Omega} \tdy \tilde\varphi_{l'}^{(0)}(\tdy)ds_{\tdy}+\Ocal(\delta^4)\\
=&-\delta^3\sum_{l'=1}^{l_0}\nabla^2\Gamma_0(\Bx-\Bz_l)\int_{\p \Omega} \tdy (\varsigma_l I - (\Kcal_\Omega^0)^*)^{-1}[\nu](\tdy)ds_{\tdy}\mathbb{H}(\Bz_l)+\Ocal(\delta^4).
\end{split}
\eeq
By using Lemma 3.6 in \cite{DLL:18}, there holds
\beq\label{eq:eq:main01pf05}
\begin{split}
&\nabla\times\Acal_{\Sigma_c}^0\left(-\frac{I}{2}+\Mcal_{\Sigma_c}^{0}\right)^{-1}\Big[\nu\times \nabla^2 \sum_{l=1}^{l_0}\Gamma_0(\cdot-\Bz_l)\mathbf{M}_l\mathbb{H}(\Bz_l)\Big](\Bx)\\
=&\nabla\Scal_{\Sigma_c}^0\left(-\frac{I}{2}+(\Kcal_{\Sigma_c}^{0})^*\right)^{-1}\Big[\nu\cdot \nabla^2 \sum_{l=1}^{l_0}\Gamma_0(\cdot-\Bz_l)\mathbf{M}_l\mathbb{H}(\Bz_l)\Big](\Bx), \quad \Bx\in \RR^3\setminus\overline{\Sigma_c}.
\end{split}
\eeq
Finally, by combing \eqnref{eq:le01sol01}, \eqnref{eq:le02sol01} and using \eqnref{eq:asymPhi01}, \eqnref{eq:main01pf04}, \eqnref{eq:eq:main01pf05} one can derive \eqref{eq:thmain01}, which completes the proof.
\end{proof}

\section{Further asymptotic analysis}
In this section, we conduct further asymptotic analysis on the leading-order term in \eqnref{eq:thmain01n}. As remarked earlier, in \eqnref{eq:thmain01}, one needs to derive the magnetic field on $\p \Sigma_c$ through the continuation of the measurement data on $\Gamma$. It is well-known that the data continuation is severely ill-conditioned. We next show that this data continuation can actually be relaxed. The result in this section may find important application in the practical realization of our theoretical findings. Before that, we first present an important axillary result.

\begin{lem}[Lemma 3.9 in \cite{DLL:18}]\label{le:04}
Let $\bZ\in \RR^3$ be fixed. Let $\Bx\in \p B_{R}$ and suppose $\|\Bz\|<R$. There holds the following asymptotic expansion
\beq\label{eq:leasm02}
\nabla \Gamma_0(\Bx-\bZ)=\sum_{n=0}^{\infty}\sum_{m=-n}^{n}
\frac{(n+1)Y_n^m(\hat\Bx)\hat\Bx-\nabla_SY_n^m(\hat\Bx)}{(2n+1)R^{n+2}}\overline{Y_n^m(\hat{\bZ})}\|\bZ\|^{n},
\eeq
where
$\hat\bZ=\bZ/\|\bZ\|$ and $\hat\Bx=\Bx/\|\Bx\|$. $Y_n^m$ is the spherical harmonics of order $m$ and degree $n$.
\end{lem}
In the sequel, we define
\beq\label{eq:leax04}
\mathbf{N}_{n+1}^{m}(\hat\Bx)=(n+1)Y_n^m(\hat\Bx)\hat\Bx-\nabla_SY_n^m(\hat\Bx),
\eeq
for $n\in \mathbb{N}\cup \{0\}$ and $m=-n, -n+1, \ldots, n-1, n$, and
\beq\label{eq:leftasmp01}
\mathbf{Q}^m_{n-1}(\hat\Bx):= \nabla_{S} Y^m_{n}(\hat{\Bx}) +n  Y^m_{n}(\hat{\Bx}) \hat\Bx, \quad \mathbf{T}_n^m(\hat\Bx):=\nabla_S Y_n^m(\hat\Bx)\times \hat\Bx
\eeq
for $n\in \mathbb{N}$ and $m=-n, -n+1, \ldots, n-1, n$. Note that $\mathbf{N}_{n+1}^m$, $\mathbf{Q}^m_{n-1}(\hat\Bx)$ and $\mathbf{T}_n^m(\hat\Bx)$ are spherical harmonics of order $n$.
Rewrite the formula \eqnref{eq:thmain01n} as
\beq\label{eq:thmain01}
(\bH^0-\bH_0^0)(\Bx)=\tilde\bH_1(\Bx)+\delta^3\tilde\bH_2(\Bx)-\delta^3\tilde\bH_3(\Bx)+\Ocal(\delta^4),
\eeq
where
$$
\widetilde\bH_1(\Bx):=\nabla\times\Acal_{\Sigma_c}^0\left(-\frac{I}{2}+\Mcal_{\Sigma_c}^{0}\right)^{-1}\Big[\nu\times(\bH^0-\bH_0^0)|_{\p \Sigma_c}^+\Big](\Bx),
$$
and
$$
\widetilde\bH_2(\Bx):=\nabla\Scal_{\Sigma_c}^0\left(-\frac{I}{2}+(\Kcal_{\Sigma_c}^{0})^*\right)^{-1}\Big[\nu\cdot \nabla^2 \sum_{l=1}^{l_0}\Gamma_0(\cdot-\Bz_l)\mathbf{M}_l\mathbb{H}(\Bz_l)\Big](\Bx),
$$
and
$$
\widetilde\bH_3(\Bx):=\sum_{l=1}^{l_0}\nabla^2 \Gamma_0(\Bx-\Bz_l)\mathbf{M}_l\mathbb{H}(\Bz_l).
$$
We next analyze the above terms one by one. First, by using Lemma \ref{le:04}, one can show that
\beq\label{eq:newassmp01}
\begin{split}
\widetilde\bH_1(\Bx)=&\int_{\p \Sigma_c}\nabla_\Bx\Gamma_0(\Bx-\By)\times \left(-\frac{I}{2}+\Mcal_{\Sigma_c}^{0}\right)^{-1}\Big[\nu\times(\bH^0-\bH_0^0)|_{\p \Sigma_c}^+\Big](\By)ds_\By\\
=&\sum_{n=0}^{\infty}\sum_{m=-n}^{n}\frac{1}{(2n+1)\|\Bx\|^{n+2}}\mathbf{N}_{n+1}^{m}(\hat\Bx)\times \\
&\int_{\p \Sigma_c}\overline{Y_n^m(\hat{\By})}\|\By\|^{n}\left(-\frac{I}{2}+\Mcal_{\Sigma_c}^{0}\right)^{-1}\Big[\nu\times(\bH^0-\bH_0^0)|_{\p \Sigma_c}^+\Big](\By)ds_\By\\
=& \frac{1}{3}\|\Bx\|^{-3}\sum_{m=-1}^{1}\mathbf{N}_{2}^{m}(\hat\Bx)\times\int_{\p \Sigma_c}\overline{Y_1^m(\hat{\By})}\|\By\|\widetilde\bH_1^{1}(\By)ds_\By+ \Ocal(\|\Bx\|^{-4}),
\end{split}
\eeq
where
\beq\label{eq:newassmp02}
\widetilde\bH_1^{1}(\By):=\left(-\frac{I}{2}+\Mcal_{\Sigma_c}^{0}\right)^{-1}\Big[\nu\times(\bH^0-\bH_0^0)|_{\p \Sigma_c}^+\Big](\By).
\eeq
Similarly, one can show that
\beq\label{eq:newassmp03}
\begin{split}
\widetilde\bH_2(\Bx)=&\int_{\p \Sigma_c}\nabla_\Bx\Gamma_0(\Bx-\By)\left(-\frac{I}{2}+(\Kcal_{\Sigma_c}^{0})^*\right)^{-1}\Big[\nu\cdot \nabla^2 \sum_{l=1}^{l_0}\Gamma_0(\cdot-\Bz_l)\mathbf{M}_l\mathbb{H}(\Bz_l)\Big](\By)ds_\By\\
=&\sum_{n=0}^{\infty}\sum_{m=-n}^{n}\frac{1}{(2n+1)\|\Bx\|^{n+2}}\mathbf{N}_{n+1}^{m}(\hat\Bx) \\
&\left(-\frac{I}{2}+(\Kcal_{\Sigma_c}^{0})^*\right)^{-1}\Big[\nu\cdot \nabla^2 \sum_{l=1}^{l_0}\Gamma_0(\cdot-\Bz_l)\mathbf{M}_l\mathbb{H}(\Bz_l)\Big](\By)ds_\By\\
=& \frac{1}{3}\|\Bx\|^{-3}\sum_{m=-1}^{1}\mathbf{N}_{2}^{m}(\hat\Bx)\int_{\p \Sigma_c}\overline{Y_1^m(\hat{\By})}\|\By\|\widetilde{H}_2^{1}(\By)ds_\By+ \Ocal(\|\Bx\|^{-4}),
\end{split}
\eeq
where
\beq\label{eq:newassmp04}
\widetilde{H}_2^{1}(\By):=\left(-\frac{I}{2}+(\Kcal_{\Sigma_c}^{0})^*\right)^{-1}\Big[\nu\cdot \nabla^2 \sum_{l=1}^{l_0}\Gamma_0(\cdot-\Bz_l)\mathbf{M}_l\mathbb{H}(\Bz_l)\Big](\By).
\eeq
To analyze $\widetilde\bH_3(\Bx)$, we first note that by differentiating both sides of \eqnref{eq:leasm02}, there holds
\beq\label{eq:leasm03}
\nabla^2 \Gamma_0(\Bx-\bZ)=\sum_{n=0}^{\infty}\sum_{m=-n}^{n}
\frac{\mathbf{A}_n^m(\hat\Bx)}{(2n+1)\|\Bx\|^{n+3}}\overline{Y_n^m(\hat{\bZ})}\|\bZ\|^{n},
\eeq
where
\beq\label{eq:leasm04}
\mathbf{A}_n^m(\hat\Bx):=(n+1)\big(\hat\Bx\nabla_SY_n^m(\hat\Bx)^T+Y_n^m(\hat\Bx)(I-\hat\Bx\hat\Bx^T)\big)-\nabla_S^2Y_n^m(\hat\Bx)-(n+2)\mathbf{N}_{n+1}^{m}(\hat\Bx)\hat\Bx^T,
\eeq
and especially when $n=0$,
\beq
\mathbf{A}_0^0(\hat\Bx)=Y_0^0(\hat\Bx)(I-3\hat\Bx\hat\Bx^T)=\frac{1}{2}\sqrt{\frac{1}{\pi}}(I-3\hat\Bx\hat\Bx^T).
\eeq
Thus one can derive that
\beq\label{eq:newassmp05}
\begin{split}
\widetilde\bH_3(\Bx)=&\sum_{l=1}^{l_0}\sum_{n=0}^{\infty}\sum_{m=-n}^{n}
\frac{\overline{Y_n^m(\hat{\bZ_l})}\|\bZ_l\|^{n}}{(2n+1)\|\Bx\|^{n+3}}\mathbf{A}_n^m(\hat\Bx)\mathbf{M}_l\mathbb{H}(\Bz_l) \\
=&\frac{1}{4\pi\|\Bx\|^{3}}\sum_{l=1}^{l_0} (I-3\hat\Bx\hat\Bx^T)\mathbf{M}_l\mathbb{H}(\Bz_l)+ \Ocal(\|\Bx\|^{-4}).
\end{split}
\eeq

Using the above results, we can further establish the following lemma.

\begin{lem}\label{le:appfur01}
Let $\mathbb{H}$ and $\mathbf{M}_l$ be defined in \eqnref{eq:asymPhi03} and \eqnref{eq:thmain02}, respectively. Then there holds
\beq\label{eq:lenewassmp01}
\frac{\delta^3}{\|\Bx\|^3}\Big(\overline{\mathbf{D}}_0\sum_{l=1}^{l_0}\mathbf{M}_l\mathbb{H}(\Bz_l)+
\Ocal(\|\Bx\|^{-2})\Big)
=\int_{\mathbb{S}}\overline{\mathbf{Q}_0(\hat\Bx)}(\bH^0-\bH_0^0)(\Bx)ds+\Ocal(\delta^4),
\eeq
where $\mathbf{D}_0$ and $\mathbf{Q}_0(\hat\Bx)$ is a $3$-by-$3$ matrix given by
\beq\label{eq:lenewassmp02}
\mathbf{D}_0:=-\frac{1}{6}\Big(\mathbf{a}_{0,1}^{0,-1},\mathbf{a}_{0,1}^{0,0},\mathbf{a}_{0,1}^{0,1}\Big)^T,
\eeq
and
\beq\label{eq:lenewassmp03}
\mathbf{Q}_0(\hat\Bx):=\Big(\mathbf{Q}_0^{-1}(\hat\Bx),\mathbf{Q}_0^{0}(\hat\Bx),\mathbf{Q}_0^{1}(\hat\Bx)\Big)^T,
\eeq
respectively. $\mathbf{a}_{n',n}^{m',m}$ is defined in \eqnref{eq:app03def09}.
\end{lem}
\begin{proof}
It can be seen that $\mathbf{Q}^m_{n-1}(\hat\Bx)$ and $\BN_{n+1}^m(\hat\Bx)$ are orthogonal to each other in $L^2(\mathbb{S})$ (see \cite{Ned,DLH19}), where $\mathbb{S}$ stands for the unit sphere, that is
\beq\label{eq:leftasmp01}
\int_\mathbb{S} \BN_{n+1}^m(\hat\Bx)\cdot\overline{\mathbf{Q}^{m'}_{n'}(\hat\Bx)}ds=0,
\eeq
for any $n, n'\in \mathbb{N}\cup \{0\}$, $m=-n, -n+1, \ldots, n-1, n$ and $m'=-n', -n'+1, \ldots, n'-1, n'$.
We show that $\mathbf{Q}^{m'}_0(\hat\Bx)$, $m'=-1, 0, 1$, are orthogonal to $\mathbf{A}_n^m(\hat\Bx)\xi$ for any $\xi\in \RR^3$ and $n\neq 0, 2$, where $\mathbf{A}_n^m(\hat\Bx)$ is defined in \eqnref{eq:leasm04}. In fact, by using \eqnref{eq:app03def13} in Appendix \ref{app:03} and the orthogonality of the vectorial spherical harmonics, one immediately obtains
\beq\label{eq:leftasmp02}
\int_\mathbb{S} \overline{\mathbf{Q}^{m'}_0(\hat\Bx)}\cdot (\mathbf{A}_n^m(\hat\Bx)\xi)ds=0 \quad \mbox{for any} \ \ n\neq 0, 2.
\eeq
Note that if $n=2$, then by \eqnref{eq:app03def13}, one has
\beq\label{eq:leftasmp03}
\begin{split}
\mathbf{A}_2^m(\hat\Bx)\xi=&\sum_{m'=m-1}^{m+1}\Big((\mathbf{c}_{1,2}^{m',m})^T\xi\mathbf{N}_{2}^{m'}(\hat\Bx)+(\mathbf{c}_{3,2}^{m',m})^T\xi\mathbf{N}_{4}^{m'}(\hat\Bx)\\
&+(\mathbf{d}_{1,2}^{m',m})^T\xi\mathbf{Q}_{0}^{m'}(\hat\Bx)+(\mathbf{d}_{3,2}^{m',m})^T\xi\mathbf{Q}_{2}^{m'}(\hat\Bx)\Big),
\end{split}
\eeq
where $\mathbf{c}_{n',n}^{m',m}$ and $\mathbf{d}_{n',n}^{m',m}$ are defined in \eqnref{eq:app03def14} and \eqnref{eq:app03def15}, respectively. Thus
\beq\label{eq:leftasmp04}
\int_\mathbb{S} \overline{\mathbf{Q}^{m'}_0(\hat\Bx)}\cdot (\mathbf{A}_2^m(\hat\Bx)\xi)ds=\frac{1}{3}(\mathbf{d}_{1,2}^{m',m})^T\xi.
\eeq
By using the orthogonality of $\mathbf{N}_3^m(\hat\Bx)$ and $Y_1^{m'}$, $m, m'\in\{-1, 0, 1\}$, one has
\beq\label{eq:leftasmp04}
\int_\mathbb{S}\overline{Y_1^{m'}(\hat\Bx)}\Big(3Y_2^m(\hat\Bx)\hat\Bx-\nabla_SY_2^m(\hat\Bx)\Big)ds=0,
\eeq
which indicates $\mathbf{a}_{1,2}^{m',m}=3 \mathbf{b}_{1,2}^{m',m}$ by using \eqnref{eq:app03def09}. Together with \eqnref{eq:leftasmp04} and \eqnref{eq:app03def15}, one thus has
\beq\label{eq:leftasmp05}
\int_\mathbb{S} \overline{\mathbf{Q}^{m'}_0(\hat\Bx)}\cdot (\mathbf{A}_2^m(\hat\Bx)\xi)ds=-\frac{1}{9}(\mathbf{a}_{1,2}^{m',m}+3\overline{\mathbf{a}_{2,1}^{m,m'}})^T\xi.
\eeq
Similarly, for $n=0$, one obtains that
\beq\label{eq:leftasmp06}
\int_\mathbb{S} \overline{\mathbf{Q}^{m'}_0(\hat\Bx)}\cdot (\mathbf{A}_0^0(\hat\Bx)\xi)ds=-\frac{1}{6}(\overline{\mathbf{a}_{0,1}^{0,m'}})^T\xi.
\eeq
Finally, by taking inner product of \eqnref{eq:thmain01} and $\mathbf{Q}^{m'}_0(\hat\Bx)$ in $L^2(\mathbb{S})$ and by using \eqnref{eq:leftasmp02}, \eqnref{eq:leftasmp05} and \eqnref{eq:leftasmp06}, one has
\beq\label{eq:leftasmp07}
\begin{split}
&\int_{\mathbb{S}}\overline{\mathbf{Q}^{m'}_0(\hat\Bx)}\cdot(\bH^0-\bH_0^0)(\Bx)ds\\
=&-\frac{1}{6}\frac{\delta^3}{\|\Bx\|^3} (\overline{\mathbf{a}_{0,1}^{0,m'}})^T\sum_{l=1}^{l_0}\mathbf{M}_l\mathbb{H}(\Bz_l)\\
&-\frac{1}{45}\sum_{m=-2}^2\frac{\delta^3}{\|\Bx\|^5}(\mathbf{a}_{1,2}^{m',m}+3\overline{\mathbf{a}_{2,1}^{m,m'}})^T
\sum_{l=1}^{l_0}\overline{Y_2^m(\hat\Bz_l)}\|\Bz_l\|^2\mathbf{M}_l\mathbb{H}(\Bz_l)+\Ocal(\delta^4).
\end{split}
\eeq
By rearranging the terms in equation \eqnref{eq:leftasmp07}, we finally arrive at \eqnref{eq:lenewassmp01}.

The proof is complete.
\end{proof}

Before proceeding further, we remark that the coefficient of higher order $\Ocal(\|\Bx\|^{-2})$ in \eqnref{eq:lenewassmp01} is accurately given in \eqnref{eq:leftasmp07}. One can readily see that if $\|\Bz_l\|/\|\Bx\|$ is small enough, that is, the radius of the measurement surface is much bigger than the length of the position vector of the magnetized anomaly, then the unknown term $\mathbf{M}_l\mathbb{H}(\Bz_l)$ is easy to recover. Thus the positions of the magnetized anomalies are easy to identify. This observation is particularly useful for a practical reconstruction procedure. For a particular case with $\Gamma=\partial B_R$, with $B_R$ a sufficiently large central ball containing $\Sigma$, we have
\begin{lem}\label{le:appfur02}
Suppose that $\Gamma\subset \partial B_R$ and let $\mathbb{H}$ and $\mathbf{M}_l$ be defined in \eqnref{eq:asymPhi03} and \eqnref{eq:thmain02}, respectively. Then there holds the following relationship,
\beq\label{eq:leappfur01}
\sum_{l=1}^{l_0}\mathbf{M}_l\mathbb{H}(\Bz_l)\approx \delta^{-3}R^3\overline{\mathbf{D}_0}^{-1}\int_{\mathbb{S}}\overline{\mathbf{Q}_0(\hat\Bx)}(\bH^0-\bH_0^0)(\Bx)ds.
\eeq
\end{lem}
\begin{proof}
By straightforward calculations, one can show that $\mathbf{D}_0$ is invertible. The proof is completed by directly using \eqnref{eq:lenewassmp01}.
\end{proof}

%One can see from \eqnref{eq:leappfur01} that in order to identify the positions of the magnetized anomalies, one needs to measure the difference of magnetic fields on the whole surface of the Earth, that is, the measurement should be taken on $\Gamma= B_R$. We want to mention that, in fact, if the measurement is taken far away from the positions of the anomalies, then the difference of the magnetic fields before and after presence of magnetic anomalies will be quite small comparing with the measurement on those places which are near the anomalies. Thus the integral in \eqnref{eq:leappfur01} can actually be calculated on a small part of $\mathbb{S}$ (limited angle which contains the anomalies).

\section{Unique recovery results for magnetic anomaly detection}

In this section, by using the asymptotic results obtained in the previous two sections, we establish the major unique recovery results for the geomagnetic detection problem \eqref{eq:geom61}. To begin with, let $D_l^{(1)}$ and $D_l^{(2)}$, $l=1, 2, \ldots, l_0$, be two sets of magnetic anomalies, which satisfy \eqnref{eq:permeab02} with $\Bz_l$ replaced by $\Bz_l^{(1)}$ and $\Bz_l^{(2)}$, respectively. Correspondingly, the material parameters  $\Gs_l$ and $\mu_l$ are replaced by $\Gs_l^{(1)}$, $\mu_l^{(1)}$ and $\Gs_l^{(2)}$, $\mu_l^{(2)}$, respectively, for $D_l^{(1)}$ and $D_l^{(2)}$, $l=1, 2, \ldots, l_0$. Let $\mathbf{H}_j$, $j=1, 2$, be the solutions to \eqref{eq:mmm1} and \eqref{eq:frequen02} with $D_l$ replaced by $D_l^{(1)}$ and $D_l^{(2)}$, respectively, and $\mathbf{H}_j^0$ be the leading-order terms of $\mathbf{H}_j$ with respect to $\omega\ll 1$. Denote by $\mathbf{M}_l^{(1)}$ and $\mathbf{M}_l^{(2)}$ the polarization tensors for $D_l^{(1)}$ and $D_l^{(2)}$, respectively, $l=1, 2, \ldots, l_0$. Furthermore, we define
\beq\label{eq:asymPhiapp01}
\mathbb{H}_j:=\nabla\times\Acal_{\Sigma_c}^0\left(-\frac{I}{2}+\Mcal_{\Sigma_c}^{0}\right)^{-1}\Big[\nu\times\bH_j^0|_{\p \Sigma_c}\Big], \quad j=1, 2.
\eeq
We first present the following axillary result.
\begin{lem}\label{le:ax00}
If the following condition is fulfilled,
\beq\label{eq:maincond01}
\nu\cdot\mathbf{H}_1^0=\nu\cdot\mathbf{H}_2^0\neq 0 \ \ {\rm on} \ \ \Gamma,
\eeq
then there hold that
\beq\label{eq:leax001}
\sum_{l=1}^{l_0}\nabla \Gamma_0(\Bx-\Bz_l^{(1)})\mathbf{M}_l^{(1)}\mathbb{H}_1(\Bz_l^{(1)})=\sum_{l=1}^{l_0}\nabla \Gamma_0(\Bx-\Bz_l^{(2)})\mathbf{M}_l^{(2)}\mathbb{H}_2(\Bz_l^{(2)}),
\eeq
for any $\Bx\in \RR^3\setminus\bigcup_{l=1}^{l_0}(\Bz_l^{(1)}\cup \Bz_l^{(2)})$, and
\beq\label{eq:leax0001}
\mathbb{H}_1=\mathbb{H}_2 \quad {\rm in} \quad \RR^3\setminus \overline{\Sigma}_c.
\eeq
\end{lem}
\begin{proof}
First, by using the unique continuation principle, one obtains that
$$\mathbf{H}_1^0=\mathbf{H}_2^0 \quad \mbox{in} \quad \RR^3\setminus\overline{\Sigma},$$
and in addition that \eqnref{eq:leax0001} holds.
Define
\beq\label{eq:leax002}
\begin{split}
\mathbb{T}_j(\Bx):=&\nabla\Scal_{\Sigma_c}^0\left(-\frac{I}{2}+(\Kcal_{\Sigma_c}^{0})^*\right)^{-1}\Big[\nu\cdot \nabla^2 \sum_{l=1}^{l_0}\Gamma_0(\cdot-\Bz_l^{(j)})\mathbf{M}_l^{(j)}\mathbb{H}_j(\Bz_l^{(j)})\Big](\Bx)\\
&-\sum_{l=1}^{l_0}\nabla^2 \Gamma_0(\Bx-\Bz_l^{(j)})\mathbf{M}_l^{(j)}\mathbb{H}_j(\Bz_l^{(j)}), \quad \Bx \in \RR^3\setminus\overline{\Sigma},\ \ j=1,2.
\end{split}
\eeq
By using \eqnref{eq:thmain01} one then has
\beq\label{eq:leax003}
\mathbb{T}_1=\mathbb{T}_2, \quad \mbox{in} \quad \RR^3\setminus\overline{\Sigma}.
\eeq
Note from \eqnref{eq:leax002} that $\mathbb{T}_j$, $j=1,2$, are both harmonic functions in $\RR^3\setminus\bigcup_{l=1}^{l_0}(\overline{\Sigma}_c\cup\Bz_l^{(1)}\cup \Bz_l^{(2)})$. By using the unique continuation again, one obtains
\beq\label{eq:leax004}
\mathbb{T}_1=\mathbb{T}_2 \quad \mbox{in} \quad \RR^3\setminus\bigcup_{l=1}^{l_0}(\overline{\Sigma}_c\cup\Bz_l^{(1)}\cup \Bz_l^{(2)}).
\eeq
By comparing the pole type of terms in $\mathbb{T}_1-\mathbb{T}_2$ one thus has
\beq\label{eq:leax005}
\sum_{l=1}^{l_0}\nabla \Gamma_0(\Bx-\Bz_l^{(1)})\mathbf{M}_l^{(1)}\mathbb{H}_1(\Bz_l^{(1)})=\sum_{l=1}^{l_0}\nabla \Gamma_0(\Bx-\Bz_l^{(2)})\mathbf{M}_l^{(2)}\mathbb{H}_2(\Bz_l^{(2)})\ \ \mbox{in}\ \ \RR^3\setminus\bigcup_{l=1}^{l_0}(\overline{\Sigma}_c\cup\Bz_l^{(1)}\cup \Bz_l^{(2)}).
\eeq
Finally the proof can be completed by using the unique continuation one more time.
\end{proof}

We are ready to present our second main result for the uniqueness in recovering multiple magnetized anomalies.

\begin{thm}\label{th:02}
If there holds
\beq\label{th:eqaxi0101}
\nu\cdot\mathbf{H}_1=\nu\cdot\mathbf{H}_2 \ \ {\rm on} \ \ \Gamma,
\eeq
with $\mathbf{H}_1^0\neq 0$ on $\Gamma$ and $\mathbb{H}_1$ defined in \eqnref{eq:asymPhiapp01} does not vanish in $\Bz_l^{(1)}$,
then $\bZ_l^{(1)}=\bZ_l^{(2)}$, $\mu_l^{(1)}=\mu_l^{(2)}$ and $\Gs_l^{(1)}=\Gs_l^{(2)}$, $l=1, 2, \ldots, l_0$.
\end{thm}
\begin{proof}
We have \eqnref{eq:leax001} by using Lemma \ref{le:ax00}. Then by following the proof of Theorem 4.2 in \cite{DLL:18} (see also Theorem 7.8 in \cite{HK07:book}), one immediately has $\bZ_l^{(1)}=\bZ_l^{(2)}$, $l=1, 2, \ldots, l_0$. Next, by using analytic extension, one obtains that
\beq\label{eq:th02pf01}
\bH_1^0\Big|_+=\bH_2^0\Big|_+ \quad \mbox{on} \quad \p D_l\cup \p \Sigma_c,
\eeq
for $l=1, 2, \ldots, l_0$.
Using the transmission condition across $\partial D_{l}$ (see Lemma 3.3 in \cite{DHU:17}), one then has
\beq\label{eq:th02pf02}
\nu\cdot\mu_0\bH_j^0\Big|_+=\nu\cdot\mu_l^{(j)}\bH_j^0\Big|_-, \, j=1, 2 \quad \mbox{on} \quad \p D_l,
\eeq
Let $\varphi_{l,j}^{(0)}$, $j=1, 2$ be the solution to \eqnref{eq:le02sol02} with $\bH^0$ be replaced by $\bH_1^0$ and $\bH_2^0$, respectively. Then by \eqnref{eq:le02sol01} and \eqnref{eq:th02pf02} one has
\beq\label{eq:th02pf03}
\Big(1-\frac{\mu_0}{\mu_{l}^{(j)}}\Big)\nu\cdot\bH_j^0\Big|_+=\nu\cdot\bH_j^0\Big|_+-\nu\cdot\bH_j^0\Big|_-=\varphi_{l,j}^{(0)}, \, j=1, 2 \quad \mbox{on} \quad \p D_l,
\eeq
Together with \eqnref{eq:asymPhi02}, \eqnref{eq:leax0001} and \eqnref{eq:th02pf01}, one further has
\beq\label{eq:th02pf04}
\Big(1-\frac{\mu_0}{\mu_{l}^{(1)}}\Big)^{-1}(\varsigma_l^{(1)} I - (\Kcal_\Omega^0)^*)^{-1}[\nu\cdot\mathbb{H}_1(\Bz_l^{(1)})]
=\Big(1-\frac{\mu_0}{\mu_{l}^{(2)}}\Big)^{-1}(\varsigma_l^{(2)} I - (\Kcal_\Omega^0)^*)^{-1}[\nu\cdot\mathbb{H}_1(\Bz_l^{(1)})],
\eeq
which implies that
\beq\label{eq:th02pf05}
(\mu_{l}^{(1)}-\mu_{l}^{(2)})(\varsigma_l^{(1)} I - (\Kcal_\Omega^0)^*)^{-1}\Big(\frac{I}{2}-(\Kcal_\Omega^0)^*\Big)(\varsigma_l^{(2)} I - (\Kcal_\Omega^0)^*)^{-1}[\nu\cdot\mathbb{H}_1(\Bz_l^{(1)})]=0.
\eeq
Note that $\int_{\p \Omega}\nu\cdot\mathbb{H}_1(\Bz_l^{(1)})=0$ and $\frac{I}{2}-(\Kcal_\Omega^0)^*$ is invertible on $L_\diamond^2(\p \Omega)$, where $L_\diamond^2(\p B)$ is the space of functions in $L^2$ and has zero average on the boundary, one thus finds from \eqnref{eq:th02pf05} that
\beq\label{eq:th02pf06}
(\mu_{l}^{(1)}-\mu_{l}^{(2)})\nu\cdot\mathbb{H}_1(\Bz_l^{(1)})=0, \quad l=1, 2, \ldots, l_0.
\eeq
Hence $\mu_{l}^{(1)}=\mu_{l}^{(2)}$ by the assumption that $\nu\cdot\mathbb{H}_1(\Bz_l^{(1)})\neq 0$. By using \eqnref{eq:le02sol01} one further has
\beq\label{eq:th02pf07}
\bH_j^0=\bH_j^0, \, j=1, 2\quad \mbox{in} \quad \RR^3\setminus\overline{\Sigma}_c.
\eeq
Next we write $\bH_1$ and $\bH_2$ in the following form:
\beq\label{eq:th02pf08}
\bH_j=\bH_j^0+\bH_j^1\omega+\Ocal(\omega^2), \quad j=1, 2.
\eeq
Define $\lambda_{l}^{(j)}:=1/(\mu_l^{(j)}\Gs_l^{(j)})$, $j=1, 2$, then by using \eqnref{eq:frequen02} one derives that
\beq\label{eq:th02pf09}
\Delta \bH_j^1(\Bx)-i(\lambda_l^{(j)})^{-1}\bH_j^0(\Bx)=0, \, j=1, 2, \quad \Bx\in\, D_l,
\eeq
which together with \eqnref{eq:th02pf07} implies that
\beq\label{eq:th02pf10}
\Delta (\bH_1^1-\bH_2^1)(\Bx)=i((\lambda_l^{(1)})^{-1}-(\lambda_l^{(2)})^{-1})\bH_1^0(\Bx), \quad \Bx\in\, D_l,
\eeq
Furthermore,
\beq\label{eq:th02pf1001}
\Delta (\bH_1^1-\bH_2^1)(\Bx)=0, \quad \Bx\in \Sigma_s\setminus\overline{\bigcup_{l=1}^{l_0}D_l}.
\eeq
By using the potential theory, one then has
\beq\label{eq:th02pf11}
(\bH_1^1-\bH_2^1)(\Bx)=i((\lambda_l^{(1)})^{-1}-(\lambda_l^{(2)})^{-1})\int_{D_l}\Gamma_0(\Bx-\By)\bH_1^0(\By)d\By, \quad \Bx\in\, D_l,
\eeq
Note that $\Delta \bH_1^0=0$ holds in $D_l$ and $\bH_1^1-\bH_2^1=0$ holds on $\p D_l$. By following exactly the same strategy as in proof of theorem 2.1 in \cite{LiuUhl15}, one can show that
$$
i((\lambda_l^{(1)})^{-1}-(\lambda_l^{(2)})^{-1})\bH_1^0(\Bx)=0, \quad \Bx\in D_l,
$$
which indicates that $\lambda_l^{(1)}=\lambda_l^{(2)}$ by the assumption that $\bH_1^0(\Bx)$ does not vanish in $\Gamma$ (so does not vanish in $D_l$).

The proof is complete.
\end{proof}

\section{Concluding remark}

In this paper, we consider the magnetic anomaly detection using the geomagnetic monitoring. It is formulated as an inverse problem associated with a magnetohydrodynamic system that stems from the widely accepted dynamo theory of geomagnetics. Under certain physically reasonable situation and practical scenario, we establish several uniqueness results in recovering the locations as well as the material parameters of the magnetized anomalies that are presented in the shell of the Earth. The mathematical strategy in establishing the unique recovery results is first based on deriving the static parts of the geomagnetic fields through the low-frequency asymptotic expansions of the geomagnetic fields, and then linearizing the inverse problem by further expanding asymptotically the static fields with respect to the anomaly size. In order to avoid dealing with the nonlinear coupled PDEs in the MHD system, we develop a data continuation strategy into our asymptotic analysis. Nevertheless, we also show how to make use of the properties of the spherical harmonic expansions to relax the severely ill-posed data continuation. We believe the strategy developed here can be extended to dealing with other coupled-physics inverse problems associated with coupled PDEs. This work is a continuation as well as an extension of our recent work \cite{DLL:18} on the identification of magnetized anomalies using geomagnetic monitoring, which aims to establish a rigorous mathematical theory for the geomagnetic detection technology. The geomagnetic model considered in \cite{DLL:18} is a linear Maxwell system, which is of different mathematical features from the magnetic system in this work.

There are several interesting and challenging topics worth further exploration. First, our theoretical study readily implies an effective reconstruction procedure. In a forthcoming paper, we shall develop efficient numerical reconstruction schemes based on our theoretical results for the magnetic anomaly detection. Second, the magnetic anomalies are presented in the shell of the Earth. The mathematical techniques developed in this work can avoid going inside the core of the Earth. However, if one intends to infer knowledge about the structure of the Earth's core through the generated magnetic field measured outside, one has to deal with the inverse problems associated with the whole MHD system. Finally, the dynamo theory has been used to interpret the magnetic fields of different celestial bodies in addition to that of the Earth. It is interesting to mathematically explore the inference of the interior structures of other planets or stars using the monitoring of their magnetic fields. We note that this technology has been used in space exploration. The mathematical study shall lay out a rigorous theoretical foundation for this kind of technology.

%In this paper, we develop a mathematical theory for identifying the magnetized anomalies using geomagnetic monitoring in MHD model, which is an extension to the Maxwell system considered in \cite{DLL:18}. We show the global uniqueness in recovering the locations of multiple magnetized anomalies as well as the parameters of the magnetized anomalies, that is, magnetic permeability and electrical conductivity. We mainly make use of the steady part in the difference of the geomagnetic fields monitored before and after the presence of the magnetized anomalies. We also show how to using the spherical harmonic estimation in order to avoid using the severely ill-posed method, i.e., unique continuation method, in recovery of the magnetic filed on the boundary of the core. The numerical implementation of identifying magnetic anomalies in MHD model will be our forth coming work.

\section*{Acknowledgment}
The work of Y. Deng was supported by NSF grant of China No. 11601528, NSF grant of Hunan No. 2017JJ3432 and No. 2018JJ3619, Innovation-Driven Project of Central South University, No. 2018CX041. The work of H. Liu was supported by the FRG and startup grants from Hong Kong Baptist University, Hong Kong RGC General Research Funds, 12302017 and 12301218.

\appendix
\section{Proof of Lemma \ref{le:asymH001}}
Define
$$N_t:=\|\Phi_s\|_{\rm{TH}(\rm{div}, \p \Sigma)}+\|\Phi_c\|_{\rm{TH}(\rm{div}, \p \Sigma_c)}+\|\nabla_{\p \Sigma}\varphi_s\|_{\rm{TH}(\rm{div}, \p \Sigma_c)}.$$
To prove Lemma \ref{le:asymH001}, we next make use of \eqnref{eq:H00}-\eqnref{eq:MM1} to conduct the asymptotic analysis with respect to $k_s$. Denote by $\Delta_{\p \Sigma}:=\nabla_{\p \Sigma}\cdot\nabla_{\p \Sigma}$. First, from \eqnref{eq:asymtmp04} one can show that
\beq\label{eq:appnewver01}
\|\Delta_{\p \Sigma} \varphi_s\|_{L^2(\p \Sigma)}=\Ocal(k_s^2N_t).
\eeq
Next, we derive the following inequality
\beq\label{eq:appnewver02}
\|\nabla_{\p \Sigma}\varphi_s\|_{{\rm TH}({\rm div}, \p \Sigma)}\leq C \|\Delta_{\p \Sigma}\varphi_s\|_{L^2(\p \Sigma)},
\eeq
where $C$ is a constant depends only on $\Sigma$. In fact, for any $\Psi\in L_T^2(\p \Sigma)$, by using the Helmholtz decomposition (\cite{ADM14, buffa2002traces}), there exists a unique $u\in H^1(\p \Sigma)$, $\int_{\p \Sigma} u ds=0$, such that 
\beq
\Psi=\nabla_{\p \Sigma} u + \nabla_{\p \Sigma} u\times \nu.
\eeq
By integration by parts there holds
\beq\label{eq:appnewver03}
\begin{split}
\|\nabla_{\p \Sigma}\varphi_s\|_{L^2(\p \Sigma)}&=\sup_{\Psi\in L_T^2(\p \Sigma)}\frac{\la \nabla_{\p \Sigma} \varphi_s, \Psi \ra}{\|\Psi\|_{{\rm TH}({\rm div},\p \Sigma)}}\\
&=\sup_{\Psi\in L_T^2(\p \Sigma)}\frac{\int_{\p \Sigma} \nabla_{\p \Sigma} \varphi_s \cdot \Big(\nabla_{\p \Sigma} u + \nabla_{\p \Sigma} u\times \nu\Big)ds}{\|\nabla_{\p \Sigma} u + \nabla_{\p \Sigma} u\times \nu\|_{L^2(\p \Sigma)}}\\
&=\sup_{\Psi\in L_T^2(\p \Sigma)}\frac{-\int_{\p \Sigma} (\Delta_{\p \Sigma} \varphi_s)  u ds}{\|\nabla_{\p \Sigma} u\|_{L^2(\p \Sigma)} + \|\nabla_{\p \Sigma} u\times \nu\|_{L^2(\p \Sigma)}}\\
&\leq \frac{1}{2}\sup_{\Psi\in L_T^2(\p \Sigma)}\frac{\|\Delta_{\p \Sigma} \varphi_s\|_{L^2(\p \Sigma)} \|u\|_{L^2(\p \Sigma)} }{\|\nabla_{\p \Sigma} u\|_{L^2(\p \Sigma)} }
\leq C \|\Delta_{\p \Sigma} \varphi_s\|_{L^2(\p \Sigma)},
\end{split}
\eeq
where the last inequality is obtained by using the Poincar\'e inequality (see, e.g., \cite{HK07:book}).
By using \eqnref{eq:appnewver01} and \eqnref{eq:appnewver02}, one thus derives that
\beq\label{eq:applem0101}
\|\nabla_{\p \Sigma}\varphi_s\|_{{\rm TH}({\rm div}, \p \Sigma)}\sim \|\Delta_{\p \Sigma}\varphi_s\|_{L^2(\p \Sigma)}=\Ocal(k_s^2N_t).
\eeq
Furthermore, from \eqnref{eq:MM0} one obtains
\beq\label{eq:applem0103}
\|\Phi_s\|_{\rm{TH}(\rm{div}, \p \Sigma)}=\Ocal(k_s^2N_t).
\eeq
By substituting \eqnref{eq:applem0101}-\eqnref{eq:applem0103} into \eqnref{eq:H00}, there holds
\beq\label{eq:applem0104}
\begin{split}
\Big(-\frac{I}{2}+\Mcal^{0}_{\Sigma_c}\Big)[\Phi_c]
=\nu\times\bH_0|_{\p \Sigma_c}^++\Ocal(k_s^2).
\end{split}
\eeq
Since $-\frac{I}{2} + \Mcal^{0}_{\Sigma_c}$ is invertible on ${\rm TH}({\rm div}, \p \Sigma_c)$, by combining \eqnref{eq:applem0101}-\eqnref{eq:applem0104}, one can obtain \eqnref{eq:le01sol01}, which completes the proof.

\section{Proof of Lemma \ref{le:asymH01}}\label{app:01}
By the continuity of $\nu\times \bH$ across $\partial \Sigma$ and $\partial D_{l}$, and using the same notation as that in Appendix A and the jump formulas \eqnref{eq:trace} and \eqnref{jumpM}, one can show that
\beq\label{eq:repre02}\begin{split}
&\Mcal_{\Sigma,\Sigma_c}^0[\Psi_0]+\Big(-\frac{I}{2}+\Mcal_{\Sigma}^0\Big)[\Phi_0]+ \Lcal_{\Sigma}^0[\nabla_{\p \Sigma}\varphi_0]+\sum_{l'=1}^{l_0}\big(\Mcal_{\Sigma, D_{l'}}^0[\Phi_{l'}]+\Lcal_{\Sigma, D_{l'}}^0[\nabla_{\p D_{l'}}\varphi_{l'}]\big)\\
=&\Mcal_{\Sigma,\Sigma_c}^{k_s}[\Psi_0]+\Big(\frac{I}{2}+\Mcal_{\Sigma}^{k_s}\Big)[\Phi_0]+ \Lcal_{\Sigma}^{k_s}[\nabla_{\p \Sigma}\varphi_0]+\sum_{l'=1}^{l_0}\big(\Mcal_{\Sigma, D_{l'}}^{k_s}[\Phi_{l'}]+\Lcal_{\Sigma, D_{l'}}^{k_s}[\nabla_{\p D_{l'}}\varphi_{l'}]\big)
\end{split}
\eeq
holds on $\p \Sigma$ and
\beq\label{eq:repre02}\begin{split}
&\Mcal_{D_l,\Sigma_c}^{k_s}[\Psi_0]+\Mcal_{D_l,\Sigma}^{k_s}[\Phi_0]+ \Lcal_{D_l, \Sigma}^{k_s}[\nabla_{\p \Sigma}\varphi_0]-\frac{\Phi_l}{2}+\sum_{l'=1}^{l_0}\big(\Mcal_{D_l, D_{l'}}^{k_s}[\Phi_{l'}]+\Lcal_{D_l, D_{l'}}^{k_s}[\nabla_{\p D_{l'}}\varphi_{l'}]\big)\\
=&\Mcal_{D_l,\Sigma_c}^{k_l}[\Psi_0]+\Mcal_{D_l,\Sigma}^{k_l}[\Phi_0]+ \Lcal_{D_l, \Sigma}^{k_l}[\nabla_{\p \Sigma}\varphi_0]+\frac{\Phi_l}{2}+\sum_{l'=1}^{l_0}\big(\Mcal_{D_l, D_{l'}}^{k_l}[\Phi_{l'}]+\Lcal_{D_l, D_{l'}}^{k_l}[\nabla_{\p D_{l'}}\varphi_{l'}]\big),
\end{split}
\eeq
holds on $\p D_l$, $l=1, 2, \ldots, l_0$.
On the other hand, by using the continuity of $\nu\cdot\mu\bH$ across $\p \Sigma$ and $\partial D_{l}$ (see Lemma 3.3 in \cite{DHU:17}), one can further show that
\beq\label{eq:pSig}
\begin{split}
&\Ncal_{\Sigma,\Sigma_c}^0[\Psi_0]+\Ncal_{\Sigma}^0[\Phi_0]+ \Big(\frac{I}{2}+(\Kcal_{\Sigma}^0)^*\Big)[\Delta_{\p \Sigma}\varphi_0]+\sum_{l'=1}^{l_0}\big(\Ncal_{\Sigma, D_{l'}}^0[\Phi_{l'}]+\Kcal_{\Sigma, D_{l'}}^0[\varphi_{l'}]\big)\\
=&\Ncal_{\Sigma,\Sigma_c}^{k_s}[\Psi_0]+\Ncal_{\Sigma}^{k_s}[\Phi_0]+ \Big(-\frac{I}{2}+(\Kcal_{\Sigma}^{k_s})^*\Big)[\Delta_{\p \Sigma}\varphi_0]+k_s^2\nu\cdot\Acal_{\Sigma}^{k_s}[\nabla_{\p \Sigma}\varphi_0]\\
&+\sum_{l'=1}^{l_0}\big(\Ncal_{\Sigma, D_{l'}}^{k_s}[\Phi_{l'}]
+\Kcal_{\Sigma, D_{l'}}^{k_s}[\varphi_{l'}]\big),
\end{split}
\eeq
holds on $\p \Sigma$ and
\beq\label{eq:pDl}
\begin{split}
&\mu_0\Big(\Ncal_{D_l,\Sigma_c}^{k_s}[\Psi_0]+\Ncal_{D_l,\Sigma}^{k_s}[\Phi_0]+ \Kcal_{D_l,\Sigma}^{k_s}[\Delta_{\p \Sigma}\varphi_0]+\frac{\varphi_l}{2}\\
&+\sum_{l'=1}^{l_0}\big(\Ncal_{D_l, D_{l'}}^{k_s}[\Phi_{l'}]+k_s^2\nu\cdot\Acal_{D_l, D_{l'}}^{k_s}[\nabla_{\p D_{l'}}\varphi_{l'}]+\Kcal_{D_l, D_{l'}}^{k_s}[\Delta_{\p D_{l'}}\varphi_{l'}]\big)\Big)\\
=&\mu_l\Big(\Ncal_{D_l,\Sigma_c}^{k_l}[\Psi_0]+\Ncal_{D_l,\Sigma}^{k_l}[\Phi_0]+ \Kcal_{D_l,\Sigma}^{k_l}[\Delta_{\p \Sigma}\varphi_0]-\frac{\varphi_l}{2}\\
&+\sum_{l'=1}^{l_0}\big(\Ncal_{D_l, D_{l'}}^{k_l}[\Phi_{l'}]+k_l^2\nu\cdot\Acal_{D_l, D_{l'}}^{k_l}[\nabla_{\p D_{l'}}\varphi_{l'}]+\Kcal_{D_l, D_{l'}}^{k_l}[\Delta_{D_{l'}}\varphi_{l'}]\big)\Big),
\end{split}
\eeq
holds on $\p D_l$,
where $\Kcal_{\Sigma', \Sigma''}^{k}: L^2(\p \Sigma')\rightarrow L^2(\p \Sigma'')$, $\Sigma', \Sigma''\in \{\Sigma, \Sigma_c, D_1, D_2, \ldots, D_{l_0}\}$, $k\in \{0, k_s, k_1, k_2, \ldots, k_{l_0}\}$, are defined by
\beq\label{eq:defKcalsp01}
\Kcal_{\Sigma', \Sigma''}^k[\varphi]:=\nu\cdot\nabla\Scal^{k}_{\Sigma''}[\varphi]|_{\partial\Sigma'}^{-}.
\eeq
If $\Sigma'=\Sigma''$, then $\Kcal_{\Sigma', \Sigma'}^k[\varphi]:=(\Kcal_{\Sigma'}^k)^*$. Similarly, $\Acal_{\Sigma', \Sigma''}^{k}: {\rm TH}({\rm div}, \p \Sigma')\rightarrow {\rm TH}({\rm div}, \p \Sigma'')$, is defined by
\beq\label{eq:defAcalsp001}
\Acal_{\Sigma', \Sigma''}^k[\Phi]:=\Acal^{k}_{\Sigma''}[\Phi]|_{\partial\Sigma'}^{-}.
\eeq
Furthermore, by using the boundary magnetic field $\nu\times\bH$ on $\p \Sigma_c$, one also has
\beq\label{eq:addmeas01}
\begin{split}
&\Big(-\frac{I}{2}+\Mcal_{\Sigma_c}^{k_s}\Big)[\Psi_0]+\Mcal_{\Sigma_c,\Sigma}^{k_s}[\Phi_0]+ \Lcal_{\Sigma_c,\Sigma}^{k_s}[\nabla_{\p \Sigma}\varphi_0]\\
&+\sum_{l'=1}^{l_0}\big(\Mcal_{\Sigma_c, D_{l'}}^{k_s}[\Phi_{l'}]+\Lcal_{\Sigma_c, D_{l'}}^{k_s}[\nabla_{\p D_{l'}}\varphi_{l'}]\big)=\nu\times\bH.\ \ \ \mbox{on}\ \p \Sigma_c
\end{split}
\eeq
By combining \eqnref{eq:repre02}-\eqnref{eq:addmeas01}, along with straightforward calculations and the similar inequality that is deduced in Appendix A, one can derive the following estimates sequentially,
\beq\label{eq:asymtH0101}
\|\Phi_0\|_{{\rm TH}({\rm div}, \p \Sigma)}=\Ocal(\omega N_f),\quad \|\nabla_{\p \Sigma}\varphi_0\|_{{\rm TH}({\rm div}, \p \Sigma)}=\Ocal(\omega N_f),\quad \|\Phi_l\|_{{\rm TH}({\rm div}, \p D_l)}=\Ocal(\omega N_f),
\eeq
where $l=1, 2, \ldots, l_0$ and $N_f:=\Ocal(\|\Psi_0\|_{{\rm TH}({\rm div}, \p \Sigma_c)}+\sum_{l'=1}^{l_0}\|\nabla_{\p D_{l'}}\varphi_{l'}\|_{{\rm TH}({\rm div}, (\p D_l)})$, and
\beq\label{eq:asymtH0102}
\begin{split}
\Psi_0=&\Big(-\frac{I}{2}+\Mcal_{\Sigma_c}^0\Big)^{-1}[\nu\times\bH]-\sum_{l'=1}^{l_0}\Big(-\frac{I}{2}+\Mcal_{\Sigma_c}^0\Big)^{-1}\Lcal_{\Sigma_c, D_{l'}}^{0}[\Delta_{D_{l'}}\varphi_{l'}]\\
&+\Ocal\Big(\omega\sum_{l'=1}^{l_0}\|\Delta_{D_{l'}}\varphi_{l'}\|_{L^2(\p D_l)}\Big).
\end{split}
\eeq
Hence
\beq\label{eq:asymtH0103}
\begin{split}
&\left(\varsigma_l I -(\Kcal_{D_l}^0)^*+\Pcal_{D_l, \Sigma_c}\Lcal_{\Sigma_c, D_l}^0\right)[\Delta_{\p D_l}\varphi_l]-\sum_{l'\neq l}^{l_0} \Big(\Kcal_{D_l, D_{l'}}^0-\Pcal_{D_l, \Sigma_c}\Lcal_{\Sigma_c, D_{l'}}^0\Big)[\Delta_{\p D_{l'}}\varphi_{l'}]\\
=&\Pcal_{D_l, \Sigma_c}[\nu\times\bH|_{\p \Sigma_c}]+\Ocal(\omega),
\end{split}
\eeq
where  the operators $\Pcal_{D_l, \Sigma_c}: {\rm TH}({\rm div}, \p \Sigma_c)\rightarrow {\rm TH}({\rm div}, \p D_l)$, $l=1, 2, \ldots, l_0$, are defined by
\beq\label{eq:defPcal01}
\Pcal_{D_l, \Sigma_c}[\Phi]:=\Ncal_{D_l, \Sigma_c}^0\Big(-\frac{I}{2}+\Mcal_{\Sigma_c}^0\Big)^{-1}[\Phi].
\eeq
Noting that $(\Kcal_{D_l}^0)^*$ and $\Pcal_{D_l, \Sigma_c}\Lcal_{\Sigma_c, D_l}^0$ are compact operators on $L^2(\p D_l)$, one can prove the invertibility of
$\varsigma_l I -(\Kcal_{D_l}^0)^*+\Pcal_{D_l, \Sigma_c}\Lcal_{\Sigma_c, D_l}^0$ on $L^2(\p D_l)$ by following a similar proof of Lemma 2.2 in \cite{DLL:18}. In fact, by the Fredholm theory, it suffices to show the uniqueness of a trivial solution to the following integral equation,
\beq\label{eq:appinvt01}
\left(\varsigma_l I -(\Kcal_{D_l}^0)^*+\Pcal_{D_l, \Sigma_c}\Lcal_{\Sigma_c, D_l}^0\right)[\Delta_{\p D_l}\varphi_l]=0.
\eeq
Note that there exists only a trial solution to the following system (see Appendix A in \cite{DLL:18})
\beq\label{eq:appinvt02}
\left\{
\begin{array}{ll}
\displaystyle{\nabla\times\bH=0, \quad \nabla\cdot\bH=0}, & \mbox{in} \,\ (\RR^3\setminus\overline{D_l \cup \Sigma_c})\cup D_l,\medskip\\
\displaystyle{\nu_l\times\bH|_+=\nu_l\times\bH|_-,}  & \mbox{on} \, \ \p D_l,\medskip \\
\displaystyle{\mu_0\nu_l\cdot\bH|_+=\mu_l\nu_l\cdot\bH|_-, }& \mbox{on} \,\ \p D_l,\medskip \\
\displaystyle{\nu\times\bH|_+=0, \quad \int_{\p \Sigma_c}\nu\cdot\bH|_+ =0,} &\mbox{on} \,\ \p \Sigma_c,\medskip \\
\displaystyle{\bH(\Bx)=\Ocal(\|\Bx\|^{-2}),  \quad \|\Bx\|\rightarrow \infty.}
\end{array}
\right.
\eeq
On the other hand, one can verify that
\beq\label{eq:appinvt03}
\bH=\Big(-\nabla\Scal_{D_l}^{0}+\nabla\times\Acal_{\Sigma_c}^0\left(-\frac{I}{2}+\Mcal_{\Sigma_c}^{0}\right)^{-1}\nu\times\nabla\Scal_{D_l}^{0}\Big)[\Delta_{\p D_l}\varphi_l]
\eeq
is also the solution to \eqnref{eq:appinvt02}. Therefore, one has
\beq\label{eq:appinvt04}
\Big(-\nabla\Scal_{D_l}^{0}+\nabla\times\Acal_{\Sigma_c}^0\left(-\frac{I}{2}+\Mcal_{\Sigma_c}^{0}\right)^{-1}\nu\times\nabla\Scal_{D_l}^{0}\Big)[\Delta_{\p D_l}\varphi_l]=0 \quad \mbox{in} \, \
\RR^3\setminus\overline{\Sigma_c}.
\eeq
Hence $\Delta_{\p D_l}\varphi_l=0$, which proves the unique trial solution to \eqnref{eq:appinvt01}. Note that $D_l$, $l=1, 2, \ldots, l_0$ are small inclusions which are disjoint from each, one can prove the unique solvability of \eqnref{eq:asymtH0103} (see Appendix B in \cite{DLL:18}).

\section{Harmonic representation of vectorial spherical polynomials}\label{app:03}
In this appendix, we shall represent $\mathbf{A}_n^m(\hat\Bx)\xi$, where $\mathbf{A}_n^m(\hat\Bx)$ is defined in \eqnref{eq:leasm04} and $\xi \in \RR^3$, in terms of vectorial spherical harmonic functions. Recall that the vectorial spherical harmonic functions of degree $n$ are composed of $\mathbf{M}_{n+1}^m(\hat\Bx)$, $\mathbf{Q}_{n-1}^m(\hat\Bx)$ and $\mathbf{T}_n^m(\hat\Bx)$, which are defined in \eqnref{eq:leax04} and \eqnref{eq:leftasmp01}.
From \eqnref{eq:leasm04} one has
\beq\label{eq:app03def02}
\begin{split}
\mathbf{A}_n^m(\hat\Bx)\xi=&(n+1)\hat\Bx\nabla_SY_n^m(\hat\Bx)^T\xi+(n+1)Y_n^m(\hat\Bx)\xi-(n+1)(n+3)Y_n^m(\hat\Bx)\hat\Bx\hat\Bx^T\xi \\
&-\nabla_S(\nabla_SY_n^m(\hat\Bx)^T\xi)+(n+2)\nabla_SY_n^m(\hat\Bx) \hat\Bx^T\xi,
\end{split}
\eeq
By vector calculus identity and integration by parts, there holds
\beq\label{eq:app03def03}
\begin{split}
\int_{\mathbb{S}} \overline{\mathbf{T}_n^m(\hat\Bx)}\cdot (\mathbf{A}_n^m(\hat\Bx)\xi) ds&=(n+1)\int_{\mathbb{S}}Y_n^m(\hat\Bx)\overline{\mathbf{T}_n^m(\hat\Bx)}\cdot\xi ds\\
&=(n+1)\int_{\mathbb{S}}\|\Bx\|^nY_n^m(\hat\Bx)\hat\Bx\cdot(\xi\times\nabla (\|\Bx\|^n \overline{Y_n^m(\hat\Bx)}))ds\\
&=(n+1)\int_{B_1}\nabla\cdot\big(\|\Bx\|^nY_n^m(\hat\Bx)\xi\times\nabla (\|\Bx\|^n \overline{Y_n^m(\hat\Bx)})\big)d\Bx\\
&=(n+1)\int_{B_1}\nabla (\|\Bx\|^n \overline{Y_n^m(\hat\Bx)})\cdot\big(\nabla(\|\Bx\|^nY_n^m(\hat\Bx))\times\xi\big)d\Bx=0,
\end{split}
\eeq
Hence, $\mathbf{A}_n^m(\hat\Bx)\xi$ is a linear combination of the spherical harmonics $\{\BN_{n+1}^m(\hat\Bx)\}$ and $\{\mathbf{Q}_{n-1}^m(\hat\Bx)\}$.
By straightforward computations, one can obtain that
\beq\label{eq:app03def04}
\begin{split}
&\int_{\mathbb{S}} \overline{\mathbf{N}_{n'+1}^{m'}(\hat\Bx)}\cdot (\hat\Bx\nabla_SY_n^m(\hat\Bx)^T\xi) ds=\frac{n'+1}{n'} \int_{\mathbb{S}} \overline{\mathbf{Q}_{n'-1}^{m'}(\hat\Bx)}\cdot (\hat\Bx\nabla_SY_n^m(\hat\Bx)^T\xi) ds\\
=&(n'+1)\int_{\mathbb{S}}\overline{Y_{n'}^{m'}(\hat\Bx)}\nabla_SY_n^m(\hat\Bx)^T\xi ds.
\end{split}
\eeq
Similarly, one can show that
\beq\label{eq:app03def05}
\begin{split}
&\int_{\mathbb{S}} \overline{\mathbf{N}_{n'+1}^{m'}(\hat\Bx)}\cdot (Y_n^m(\hat\Bx)\xi) ds\\
=&-\int_{\mathbb{S}}Y_n^m(\hat\Bx)\nabla_S\overline{Y_{n'}^{m'}(\hat\Bx)}^T\xi ds+(n'+1)\int_{\mathbb{S}}\overline{Y_{n'}^{m'}(\hat\Bx)}Y_n^m(\hat\Bx)\hat\Bx^T\xi ds,
\end{split}
\eeq
and
\beq\label{eq:app03def051}
\begin{split}
&\int_{\mathbb{S}} \overline{\mathbf{Q}_{n'-1}^{m'}(\hat\Bx)}\cdot (Y_n^m(\hat\Bx)\xi) ds\\
=&\int_{\mathbb{S}}Y_n^m(\hat\Bx)\nabla_S\overline{Y_{n'}^{m'}(\hat\Bx)}^T\xi ds+n'\int_{\mathbb{S}}\overline{Y_{n'}^{m'}(\hat\Bx)}Y_n^m(\hat\Bx)\hat\Bx^T\xi ds,
\end{split}
\eeq
and
\beq\label{eq:app03def06}
\begin{split}
&\int_{\mathbb{S}} \overline{\mathbf{N}_{n'+1}^{m'}(\hat\Bx)}\cdot (Y_n^m(\hat\Bx)\hat\Bx\hat\Bx^T\xi) ds=\frac{n'+1}{n'} \int_{\mathbb{S}} \overline{\mathbf{Q}_{n'-1}^{m'}(\hat\Bx)}\cdot (Y_n^m(\hat\Bx)\hat\Bx\hat\Bx^T\xi) ds\\
=&(n'+1)\int_{\mathbb{S}}\overline{Y_{n'}^{m'}(\hat\Bx)}Y_n^m(\hat\Bx)\hat\Bx^T\xi ds.
\end{split}
\eeq
Next, for the last two terms in \eqnref{eq:app03def02}, using integration by parts, we have
\beq\label{eq:app03def07}
\begin{split}
&\int_{\mathbb{S}} \overline{\mathbf{N}_{n'+1}^{m'}(\hat\Bx)}\cdot \nabla_S(\nabla_SY_n^m(\hat\Bx)^T\xi) ds=-\int_{\mathbb{S}} \overline{\mathbf{Q}_{n'-1}^{m'}(\hat\Bx)}\cdot \nabla_S(\nabla_SY_n^m(\hat\Bx)^T\xi) ds\\
=&-\int_{\mathbb{S}} \nabla_S\overline{Y_{n'}^{m'}(\hat\Bx)}\cdot \nabla_S(\nabla_SY_n^m(\hat\Bx)^T\xi)ds\\
=&\int_{\mathbb{S}} \Delta_S\overline{Y_{n'}^{m'}(\hat\Bx)}\nabla_SY_n^m(\hat\Bx)^T\xi ds=-n'(n'+1)\int_{\mathbb{S}}\overline{Y_{n'}^{m'}(\hat\Bx)}\nabla_SY_n^m(\hat\Bx)^T\xi ds.
\end{split}
\eeq
Furthermore, there holds
\beq\label{eq:app03def08}
\begin{split}
&\int_{\mathbb{S}} \overline{\mathbf{N}_{n'+1}^{m'}(\hat\Bx)}\cdot (\nabla_SY_n^m(\hat\Bx) \hat\Bx^T\xi) ds=-\int_{\mathbb{S}} \overline{\mathbf{Q}_{n'-1}^{m'}(\hat\Bx)}\cdot (\nabla_SY_n^m(\hat\Bx) \hat\Bx^T\xi) ds\\
=&-\int_{\mathbb{S}} \nabla_S\overline{Y_{n'}^{m'}(\hat\Bx)}\cdot (\nabla_SY_n^m(\hat\Bx) \hat\Bx^T\xi)ds\\
=&-\int_{\mathbb{S}} \nabla_S\overline{Y_{n'}^{m'}(\hat\Bx)}\cdot \big(\nabla (Y_n^m(\hat\Bx) \hat\Bx^T\xi)-Y_n^m(\hat\Bx)\xi\big)ds\\
=&-n'(n'+1)\int_{\mathbb{S}}\overline{Y_{n'}^{m'}(\hat\Bx)}Y_n^m(\hat\Bx) \hat\Bx^T\xi ds+\int_{\mathbb{S}} Y_n^m(\hat\Bx)\nabla_S\overline{Y_{n'}^{m'}(\hat\Bx)}^T\xi ds.
\end{split}
\eeq
In the following, we define
\beq\label{eq:app03def09}
\mathbf{a}_{n',n}^{m',m}:=\int_{\mathbb{S}}\overline{Y_{n'}^{m'}(\hat\Bx)}\nabla_SY_n^m(\hat\Bx) ds,\quad \mathbf{b}_{n',n}^{m',m}:=\int_{\mathbb{S}}\overline{Y_{n'}^{m'}(\hat\Bx)}Y_n^m(\hat\Bx)\hat\Bx ds.
\eeq
By combining \eqnref{eq:app03def04}-\eqnref{eq:app03def08} and using \eqnref{eq:app03def02}, one thus has
\beq\label{eq:app03def10}
\begin{split}
&\int_{\mathbb{S}} \overline{\mathbf{N}_{n'+1}^{m'}(\hat\Bx)}\cdot (\mathbf{A}_n^m(\hat\Bx)\xi) ds\\
=&\left((n'+1)(n'+n+1)\mathbf{a}_{n',n}^{m',m}-(n'+1)(n'+n+1)(n+2)\mathbf{b}_{n',n}^{m',m}+\overline{\mathbf{a}_{n,n'}^{m,m'}}\right)^T\xi,
\end{split}
\eeq
and
\beq\label{eq:app03def11}
\begin{split}
&\int_{\mathbb{S}} \overline{\mathbf{Q}_{n'+1}^{m'}(\hat\Bx)}\cdot (\mathbf{A}_n^m(\hat\Bx)\xi) ds=\left(n'(n-n')\mathbf{a}_{n',n}^{m',m}-n'(n-n')(n+2)\mathbf{b}_{n',n}^{m',m}-\overline{\mathbf{a}_{n,n'}^{m,m'}}\right)^T\xi.
\end{split}
\eeq
By using the following elementary result (cf. \cite{Ned})
\beq\label{eq:app03def12}
\mathbf{a}_{n',n}^{m',m}=\mathbf{b}_{n',n}^{m',m}=0, \quad \mbox{for any} \, n'\neq n-1, \, n+1, \mbox{and} \, m'\neq m-1, m, m+1,
\eeq
one finally obtains
\beq\label{eq:app03def13}
\begin{split}
\mathbf{A}_n^m(\hat\Bx)\xi=&\sum_{m'=m-1}^{m+1}\Big((\mathbf{c}_{n-1,n}^{m',m})^T\xi\mathbf{N}_{n}^{m'}(\hat\Bx)+(\mathbf{c}_{n+1,n}^{m',m})^T\xi\mathbf{N}_{n+2}^{m'}(\hat\Bx)\\
&+(\mathbf{d}_{n-1,n}^{m',m})^T\xi\mathbf{Q}_{n-2}^{m'}(\hat\Bx)+(\mathbf{d}_{n+1,n}^{m',m})^T\xi\mathbf{Q}_{n}^{m'}(\hat\Bx)\Big),
\end{split}
\eeq
where
\beq\label{eq:app03def14}
\mathbf{c}_{n',n}^{m',m}:=\frac{(n'+1)(n'+n+1)\mathbf{a}_{n',n}^{m',m}-(n'+1)(n'+n+1)(n+2)\mathbf{b}_{n',n}^{m',m}+\overline{\mathbf{a}_{n,n'}^{m,m'}}}{(n'+1)(2n'+1)},
\eeq
and
\beq\label{eq:app03def15}
\mathbf{d}_{n',n}^{m',m}:=\frac{n'(n-n')\mathbf{a}_{n',n}^{m',m}-n'(n-n')(n+2)\mathbf{b}_{n',n}^{m',m}-\overline{\mathbf{a}_{n,n'}^{m,m'}}}{n'(2n'+1)}.
\eeq

\end{document}